\documentclass[reqno, 11pt, a4paper]{amsart}
\usepackage{latexsym,ifthen,xspace}
\usepackage{amsmath,amssymb,amsthm}
\usepackage{enumerate, calc, bbm}
\usepackage{subfigure}
\usepackage[usenames,dvipsnames]{xcolor}
\usepackage{paralist}
\usepackage[colorlinks=true, pdfstartview=FitV, linkcolor=blue,
  citecolor=blue, urlcolor=blue, pagebackref=false]{hyperref}
\usepackage[text={33pc,605pt},centering]{geometry}    %11pt
\usepackage{graphicx}
\newtheorem{theorem}{Theorem}[section]

\newtheorem{prop}[theorem]{Proposition}
\newtheorem{assumption}[theorem]{Assumption}

\theoremstyle{definition}
\newtheorem{definition}[theorem]{Definition}
\newtheorem{example}[theorem]{Example}

\theoremstyle{remark}
\newtheorem{remark}[theorem]{Remark}

\numberwithin{equation}{section}

\DeclareMathAlphabet{\mathsl}{OT1}{cmss}{m}{sl}
\SetMathAlphabet{\mathsl}{bold}{OT1}{cmss}{bx}{sl}

\newcommand{\al}{\ensuremath{\alpha}}

\newcommand{\ga}{\ensuremath{\gamma}}
\newcommand{\de}{\ensuremath{\delta}}

\newcommand{\ka}{\ensuremath{\kappa}}

\newcommand{\si}{\ensuremath{\sigma}}

\newcommand{\om}{\ensuremath{\omega}}

\newcommand{\Si}{\ensuremath{\Sigma}}

\newcommand{\Om}{\ensuremath{\Omega}}
\newcommand{\cB}{\ensuremath{\mathcal B}}
\newcommand{\cC}{\ensuremath{\mathcal C}}

\newcommand{\cF}{\ensuremath{\mathcal F}}

\newcommand{\cL}{\ensuremath{\mathcal L}}

\newcommand{\cP}{\ensuremath{\mathcal P}}

\newcommand{\bbN}{\ensuremath{\mathbb N}}

\newcommand{\bbP}{\ensuremath{\mathbb P}}

\newcommand{\bbR}{\ensuremath{\mathbb R}}

\newcommand{\bbZ}{\ensuremath{\mathbb Z}}
\newcommand{\md}{\ensuremath{\mathrm{d}}}
\newcommand{\mD}{\ensuremath{\mathrm{D}}}
\newcommand{\mm}[1]{\ensuremath{\scriptscriptstyle #1}}

\newcommand{\Norm}[2]{%
  \ensuremath{%
    \mathchoice{\big\lVert #1 \big\rVert}
     {\lVert #1 \rVert}
     {\lVert #1 \rVert}
     {\lVert #1 \rVert}_{\raisebox{-.0ex}{$\scriptstyle #2$}}
  }
}

\DeclareMathOperator{\mean}{\mathbb{E}}

\DeclareMathOperator{\Mean}{\mathrm{E}}
\DeclareMathOperator{\prob}{\mathbb{P}} %law random environment
\DeclareMathOperator{\Prob}{\mathrm{P}} %law random walk

\newcommand{\ldef}{\ensuremath{\mathrel{\mathop:}=}}

\newcommand{\indicator}{\mathbbm{1}}

\begin{document}

\title{Homogenization theory of random walks in degenerate random environment}

%    Remove any unused author tags.

%    author one information
\author{Sebastian Andres}
\address{Technische Universit\"at Braunschweig}
\curraddr{Institut f\"ur Mathematische Stochastik, Universit\"atsplatz 2, 38106 Braunschweig}
\email{sebastian.andres@tu-braunschweig.de}
\thanks{}

%    author two information
% \author{Martin Slowik}
% \address{University of Mannheim}
% \curraddr{Mathematical Institute, B6, 26, 68159 Mannheim}
% \email{slowik@math.uni-mannheim.de\\ https://orcid.org/0000-0001-5373-5754}
% \thanks{}



\date{\today}

\dedicatory{}

\begin{abstract}
  Recent progress on the understanding of the Random Conductance Model is reviewed. A particular emphasis is on homogenization results such as functional central limit theorems, local limit theorems and heat kernel estimates for almost every realization of the environment, established for random walks among stationary ergodic conductances that are possibly unbounded but satisfy certain moment conditions. 
\end{abstract}

\maketitle

%\tableofcontents

\section{Introduction}\label{sec:INTRO}
\subsection{Background}
One main motivation for the study of random walks in random environment is to explore the relationships between the microscopic properties of a disordered medium and its large scale conductivity. Examples include heat conduction in porous media or composite materials, electrical conducting properties in metals with impurities, to name just a few.
One particular question of interest is whether homogenization occurs, in the sense that on macroscopic scales the irregular structures in the random medium are ‘smoothed out’  and only an averaged effect remains; the alternative being 
that an effect of the irregularities persists even at macroscopic scales leading to an anomalous behaviour.
However, the relation between the microscopic properties of the medium and its effective ones is subtle and highly sensitive to the geometrical, spatial and statistical properties of the underlying
medium. So it is natural to ask what kind of statistical properties of the random medium ensure a homogenization of the rescaled random walker on macroscopic scales.

The mathematical rigorous analysis of random walks in random environments has been at the centre of interest in probability theory since the early 1980s,  see e.g.\ the reviews \cite{Hu96, Sz06, Ze04}. During the last decades there has been a major research effort  focused on the derivation of heat kernel estimates and functional central limit theorems (a.k.a.\ invariance principles), i.e.\ the convergence of the random walk towards Brownian motion under the diffusive rescaling of space and time. In this article we will review such homogenization results by means of the random conductance model, one the most prominent models for a random walk in random environment. On the other hand, depending on the properties of the environment, sub-diffusive behaviour may occur, which is typical for diffusion processes on fractal spaces, see e.g.\ \cite{Ba98, Ku14}.  One manifestation of such irregular behaviour is the presence of heat kernel fluctuations, appearing in many instances of stochastic processes on fractals and in random media. In particular, for processes in low dimensions and models at criticality, heat kernel fluctuations are known to occur, caused by local irregularities in the random medium, see \cite{ACK23} for a recent review on this topic.

\subsection{The model}
 The \emph{random conductance model} (RCM) is a particular class of reversible random walks in random environments defined as follows. For $d\geq 2$ consider  the Euclidean lattice $(\mathbb{Z}^d, E_d)$, where $E_d=\{e=\{x,y\}: x,y\in\bbZ^d, |x-y|=1\}$ denotes the set of non-oriented nearest neighbour edges. We call two vertices $x,y \in \bbZ^d$ adjacent if $\{x,y\} \in E_d$, and we then write $x \sim y$.
Each edge  $\{x,y\}\in  E_d$ is endowed with a non-negative weight $\om(\{x,y\})$, also called conductance. We write $\om(x,y)=\om(\{x,y\})=\om(y,x)$, 
and $\om(x,y)=0$ if $\{x, y\} \not\in E_d$.
We are particularly interested in the situation when the collection $\{\omega(e): e \in E_d\}$ is itself random and distributed according to a stationary and ergodic law $\prob$ on the measurable space $(\Omega, \mathcal{F}) \ldef ([0, \infty)^{E_{d}}, \mathcal{B}([0, \infty))^{\otimes E_{d}})$ equipped with the Borel-$\sigma$-algebra.
For a fixed $\om \in \Omega$ set 
\begin{align*}  % \label{jumpP}
  \mu^\om(x) \ldef  \sum_{y\sim x} \om(x,y), \qquad
 p^\om(x,y) \ldef \frac{\om(x,y)}{\mu^\om(x)}. 
\end{align*}
We endow $(\Omega, \mathcal{F})$ with the $d$-parameter group of space shifts $(\tau_{x})_{x \in \mathbb{Z}^{d}}$ acting on $\Omega$ as
\begin{align*} %\label{eq:def:space_shift}
  (\tau_{z}\, \omega)(\{x, y\})
  \;\ldef\;
  \omega(\{x+z, y+z\}),
  \qquad \forall\, \{x, y\} \in E_{d}.
\end{align*}

We will study continuous time random walks on $\bbZ^d$ which 
jump according to the probabilities $p^\om(x,y)$.
There are two natural choices of this. The first one is called the
\emph{variable speed random walk} (VSRW) $X = (X_t\!: t \geq 0 )$, which waits at $x$ an exponential time with mean $1/\mu^\om(x)$, with generator
\[
 \big(\cL_X^{\omega} f)(x)= \sum_{y \sim x} \omega(x,y) \, \left(f(y)-f(x)\right).
\]
It is a reversible Markov chain with symmetrising measure given by the counting measure. We denote by $\Prob_x^\om$ the law of the process starting at the vertex $x \in \bbZ^d$.  The corresponding expectation will be denoted by $\Mean_x^\om$.  
The second walk, called the  \emph{constant speed random walk} (CSRW), denoted by $Y = (Y_t\!: t \geq 0 )$,  waits at $x$ for an  exponential time 
with mean $1$, with generator 
\begin{align*} 
  \big(\cL_Y^{\om}\, f\big)(x)
  \;=\;
  \frac 1 {\mu^\om(x)} \sum_{y \sim x} \om(x,y) \, \big(f(y) - f(x)\big).
\end{align*}
If $\mu^\om(x)=0$ we write $\cL_X f(x)=\cL_Y f(x)=0$.
The CSRW is reversible with respect to the measure $\mu^\om$. Note that $X$ and $Y$ are time-changes of each other. 
The CSRW is mainly of probabilistic interest due to its resemblence to the corresponding random walk in discrete time. On the other hand,  VSRWs and their generators are objects of study in stochastic homogenization in PDE theory, and they are also relevant due to their appearance in statistical mechanics models.  
We denote by $p^\omega(t,x,y)$ and $q^{\om}(t,x,y)$, $x, y \in \bbZ^d$, $t \geq 0$, the transition density with respect to the reversible measure (or \emph{heat kernels}) of the VSRW and CSRW, respectively, i.e.\
\begin{align*}
  p^\omega(t,x,y) \ldef \Prob^\omega_x(X_t=y), \qquad  q^{\om}(t,x,y) \ldef \frac{\Prob_x^\om\big(Y_t = y\big)}{\mu^{\om}(y)}.
\end{align*}

As an important example for an environment, consider the case when the conductances are independent and identically distributed (i.i.d.). Note that $X$ or $Y$ never jump across $e$ whenever $\om(e)=0$. So if $\prob( \om(e)>0)$ is less 
than $p_c\in (0,1)$, the critical threshold for Bernoulli bond percolation on $\bbZ^d$, then both $X$ and $Y$  are $\prob$-a.s.\ confined to a finite set. Thus it is very natural to assume that $\prob( \om(e) >0 ) > p_c$.
Then, it is well known that there exists a $\prob$-a.s.\ unique infinite cluster $\cC_\infty$ of connected open edges, where an edge $e$ is called open if $\om(e)>0$. We denote by $\prob_0[ \,\cdot\, ] \ldef \prob\bigl[ \,\cdot \,|\, 0 \in \mathcal{C}_{\infty}\bigr]$ the conditional distribution given the event $\{0 \in \mathcal{C}_{\infty}\}$.

\subsection{Questions and aims}

In the late 1980s, the pioneering works \cite{dMFGW89} and \cite{KV86} already proved that a functional central limit theorem holds under the \emph{averaged}, or \emph{annealed}, law (i.e.\ the joint law of the environment and the process). The results also imply an `invariance principle in probability'. On the other hand, much effort in the past 20 years has been invested in the derivation of the following \emph{quenched} statements for the random walks and their heat kernel, holding for almost all realisations of the environment (for simplicity  only stated formally for the VSRW here).

\begin{itemize}
\item \textbf{Quenched functional central limit theorem (QFCLT).} For almost all realisations of the environment, the rescaled random walk $( n^{-1} X_{n^ 2 t})_{t\geq 0}$ converges in law towards a Brownian motion on $\mathbb{R}^ d$ with deterministic covariance matrix $\Sigma^2$.

\item  \textbf{Quenched local central limit theorem.}  For almost all realisations of the environment, the transition kernel of the random walk converges under suitable rescaling towards the Gaussian transition kernel of Brownian motion with some deterministic and non-degenerate covariance matrix $\Sigma^2$, i.e.\
$$n^d p^\omega (n^2t, 0,\lfloor nx\rfloor)  \underset{n\to \infty} {\longrightarrow} \bar p^{\Sigma}(t,0,x),$$
uniformly over compact sets, where 
\begin{align} \label{eq:GaussHK}
  \bar p^\Sigma(t,x,y)
  \;\ldef\;
  \frac{1}{\sqrt{(2\pi t)^d \det \Si^2}}
  \exp\!\left(-\frac{(x-y) \cdot(\Si^2)^{-1}(x-y)}{2t}\right),
\end{align}
and for any $x\in \bbR^d$ we write  $\lfloor nx\rfloor$ for a closest point to $nx$ in the underlying graph w.r.t.\ the $|\cdot|_\infty$-norm. 
\item \textbf{Gaussian heat kernel estimates.} For almost all realisations of the environment, we have
 $ p^\omega(t,x,y) \leq c \, t^{-d/2} \exp\big(-c  |x-y|^2/t\big)$,
and a similar lower bound, for sufficiently large $t$ and $|x-y|$.
\end{itemize}

Since the heat kernel is a fundamental solution of the heat equation $\partial_t u - \cL_X^\omega u=0$, a closely related, purely analytic statement is the parabolic Harnack inequality. In fact, since the 1960s Harnack inequalities have had a big influence on PDE theory and differential geometry, in particular they provide control on the oscillations and  imply H\"older regularity of harmonic functions.

\medskip
In the special case of i.i.d.\ environments (discussed in more detail in Section~\ref{sec:iid} below), the QFCLT does hold in general once $\mathbb{P}[\omega(e)>0]>p_c$.  On the other hand, 
local limit theorems are stronger statements, which are rare and difficult to prove for Markov chains in general. In fact, already in the i.i.d.\ case, where the QFCLT does hold, we see the surprising effect that due to a trapping phenomenon (caused by the presence of very small or very large conductances)  the heat kernel may behave sub-diffusively  \cite{BBHK08}, in particular a quenched local limit theorem and Gaussian heat kernel bounds may fail in general. Nevertheless, they do hold, for instance, in the case of uniformly elliptic conductances, where $\mathbb{P}(c_1 \leq \omega_e \leq c_2)=1$ for some $0<c_1<c_2<\infty$, or for  random walks on i.i.d.\ supercritical percolation clusters (see \cite{De99, Ba04, BH09}).

Invariance principles and local limit theorems for random walks in random environment can be regarded as a probabilistic counterpart to the area of stochastic homogenization in PDE theory, initiated in the celebrated works of Kozlov \cite{Ko79} and Papanicolaou-Varadhan \cite{PV81}. During the last decade, starting from the works \cite{GO11,GO12}, the area of quantitative stochastic homogenization of elliptic and parabolic equations has witnessed tremendous progress, see e.g.\ \cite{AD18, armstrong2016b, armstrong2016, armstrong2014,BellaFFOtto2016,BellaFO2016,DG21,DGO20,GNO15,GNOreg,GO15,GO14, BK24}, and the two monographs \cite{AK22,AKM19}.

In this article, supplementing the already existing surveys \cite{Bi11, Ku14}, our focus will be on the more qualitative probabilistic questions concerning the random walk and its heat kernel listed above, in particular for RCMs in general ergodic environments only required to satisfy certain moment conditions. In contrast to the quantitative homogenization results mentioned above, most of the main results discussed in this article do not require any assumptions on the correlation decay of the environment such as mixing assumptions, for instance.  

The rest of the paper is organised as follows. First, we summarize the situation in the i.i.d.\ case in Section~\ref{sec:iid}. Then, in Section~\ref{sec:ergodic} we will consider general ergodic environments satisfying certain moment conditions, where we will also provide a detailed overview of the proof of the QFCLT in this setting. Homogenization of RCMs with time-dependent conductances will be discussed in Section~\ref{sec:dynamic}. Finally, in Section~\ref{sec:longrange} we will review recent results for RCMs with long-range jumps.

Throughout the paper we write $c$ to denote a positive constant which may change on each appearance. Constants denoted $c_{i}$ will be the same through the paper.

\section{The i.i.d.\ case} \label{sec:iid}
In this section we consider the special case of i.i.d.\ conductances before we will discuss more general ergodic environments from the next section on. That is, throughout this section we assume that the random environment is given by i.i.d.\ random variables $(\om(e), e \in  E_d)$  taking values in $[0, \infty)$.

\subsection{Quenched invariance principle and other scaling limits}
We start by discussing the QFCLT, which has been studied by a number of different authors under various restrictions on the law of the environment about 15-20 years ago, see \cite{SS04, BB07, MP07,BP07, BD10},  with all approaches synthesized together in \cite{ABDH13} into the following result.

\begin{theorem}[\cite{ABDH13}]\label{thm:QFCLTiid} Let $d \ge 2$ and 
suppose that $(\om(e), e\in E_d)$ are i.i.d.,  $\om(e)\geq 0$ $\prob$-a.s.\  and $\prob( \om(e) >0 ) > p_c$.

\begin{enumerate}[(i)]
\item Let $X$ be the VSRW and set $X^{(n)}\ldef  (n^{-1} X_{n^2t}: \, t\geq 0)$, $n \in \bbN$. Then, 
$\prob_0$-a.s. 
$X^{(n)}$ converges, under $\Prob_0^\om$,
in law to a Brownian motion on $\bbR^d$ with covariance matrix
$\sigma_X^2 I$, where $\sigma_X>0$ is non-random. (Here  $I$ denotes the identity matrix.)

\item Let $Y$ be the CSRW and set $Y^{(n)}\ldef  (n^{-1} Y_{n^2t}:\, t\geq 0)$, $n \in \bbN$. Then, $\prob_0$-a.s. $Y^{(n)}$ converges, 
under $\Prob_0^\om$,
in law to a Brownian motion on $\bbR^d$ with covariance matrix
$\sigma_Y^2 I$, where    
\begin{equation*}
 \sigma_Y^2  = 
\begin{cases}
%\sigma_X^2 /(2d \mean[\om(e)]), & \hbox{ if } \mean[\om(e)]<\infty,     \\
\sigma_X^2 /\mean_0[\mu^\om(0)], & \hbox{ if } \mean[\om(e)]<\infty,     \\
0, &  \hbox{ if } \mean[\om(e)] = \infty.
\end{cases}
\end{equation*}
\end{enumerate}
\end{theorem}

\renewcommand{\thesubfigure}{}
\begin{figure} \label{fig:traps}
	\begin{center}
%\fbox{	
\subfigure[(a)]{ \scalebox{0.3}{   \includegraphics{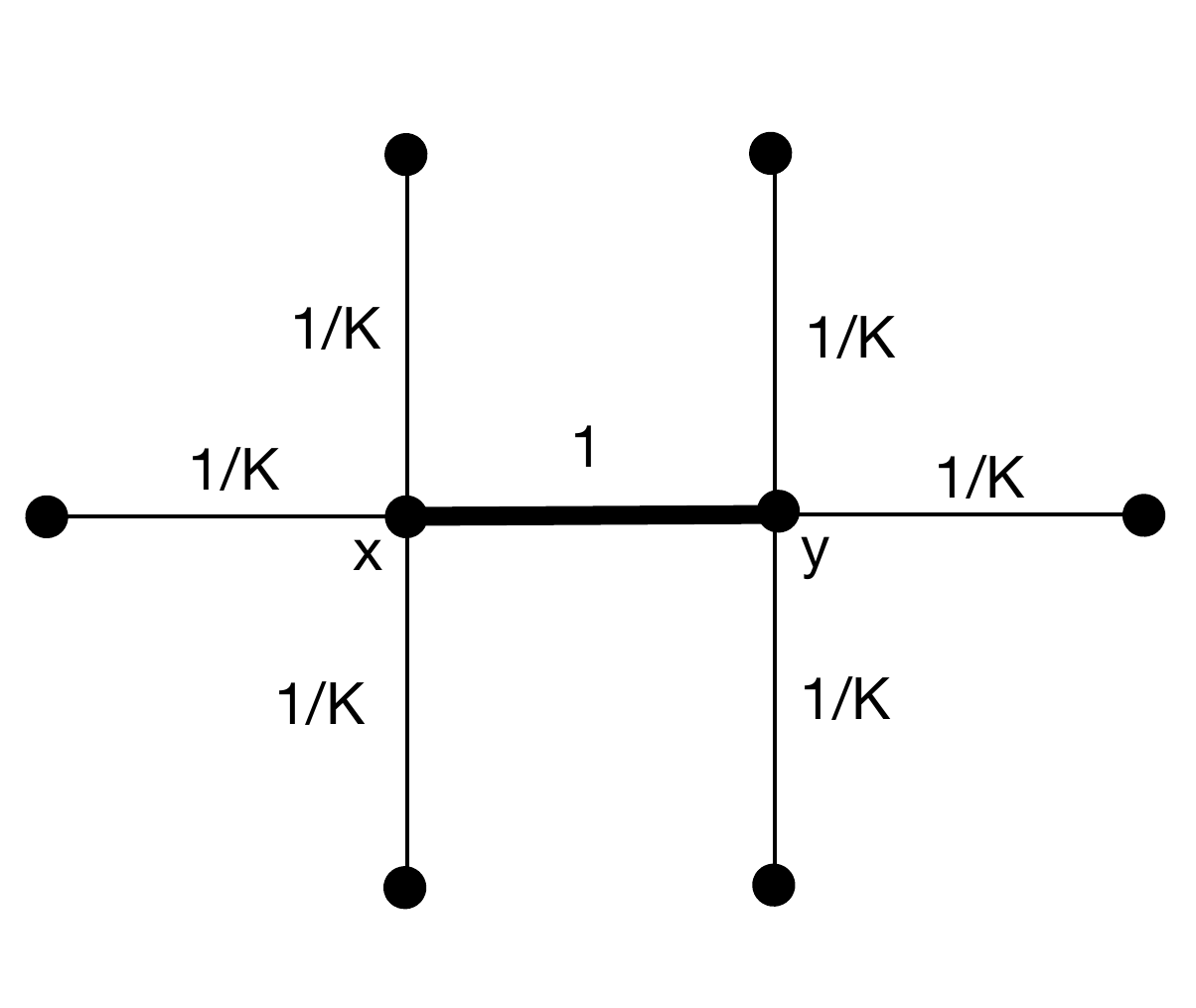} } }
\subfigure[(b)]{ \scalebox{0.3}{   \includegraphics{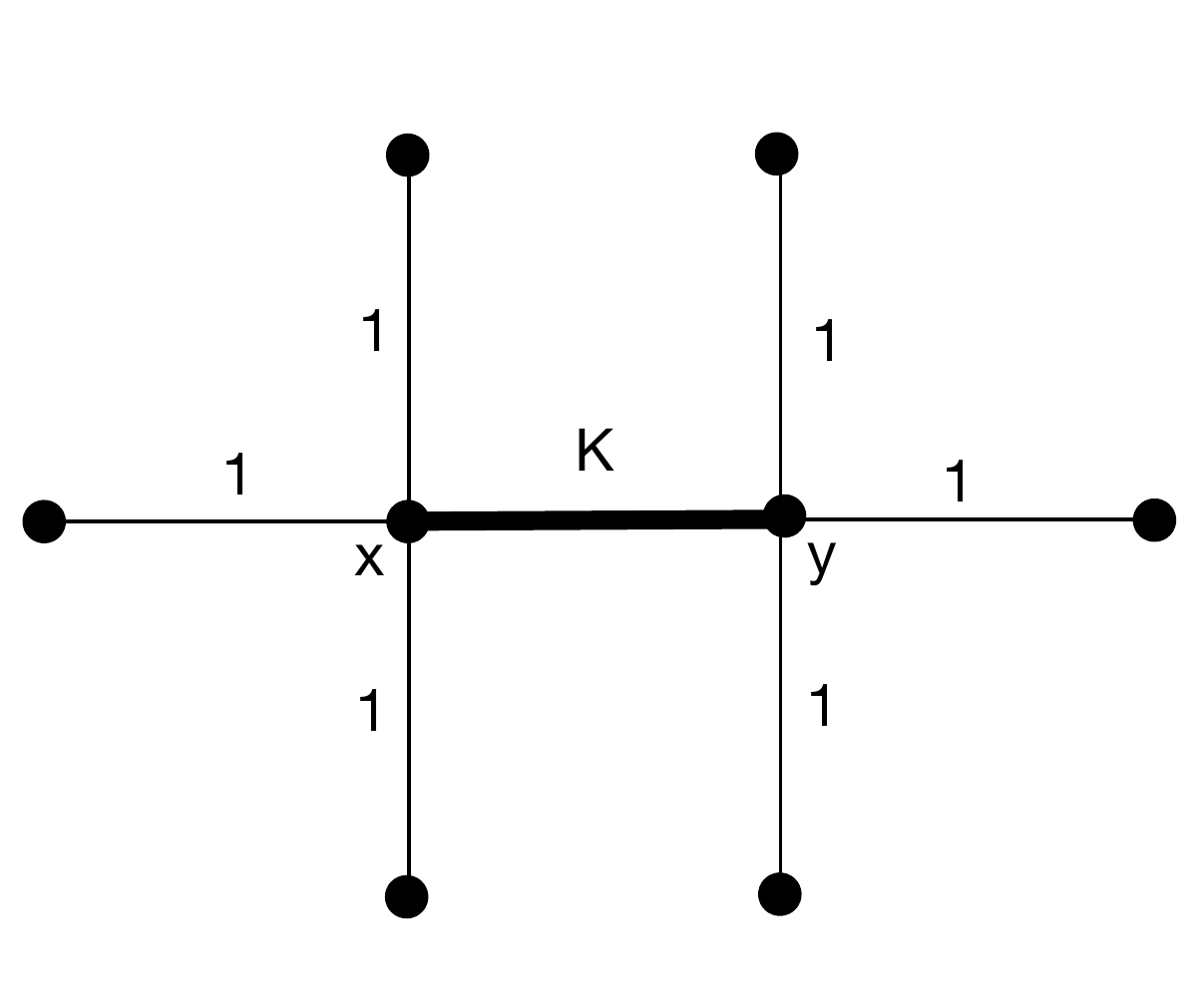} } }
%}
\end{center}
\caption{(a) Trap of the first kind and (b) Trap of the second kind}
\label{fig:traps}
\end{figure}

The main difficulty in studying the general RCM is the possibility of `traps', which may be due to either edges with small positive conductance, or very large conductance (Figure~\ref{fig:traps}). 

\
For the first kind of trap, consider an edge  $\{x,y\}$ with $\om(x,y)=1$, 
separated by edges of conductance $1/K$ with $K\gg 1$ from a path of edges of conductance $1$. Once the walk enters the trap, incurring a cost $1/K$ of probability, it will stay bouncing back and forth for time 
 $O(K)$ with a uniformly positive probability. This kind of trap is capable of capturing both CSRW and VSRW for time  $O(K)$.

The second kind of trap is associated with an edge $\{x,y\}$ with
$\om(x,y) = K \gg1 $, and with $\om(e)=O(1)$ for all other edges $e$ having an endpoint in $\{x,y\}$. 
In this case the CSRW will be trapped in the set
$\{x,y\}$ for time $O(K)$, but the VSRW will not be trapped as its waiting time at each visit to $x$ or $y$ will be $O(1/K)$. This
explains why  in the situation when $\mean[\om(e)] = \infty $, 
the VSRW and CRSW do have quite different long time behaviour. In particular
due to the traps of the second kind the limiting variance of the CSRW  vanishes, meaning that the limiting process is zero,  and it is therefore natural to ask further about the behaviour of the CSRW. 

Indeed in \cite{BC11, Ce11} it has been proved that for i.i.d. nearest neighbour conductances with $\al$-stable upper tail, $\al\in (0,1)$,
i.e.\  $\prob[ \om(e) > t] \sim t^{-\alpha}$, the rescaled process $ ( n^{-\al} X_{n^2t})_{t\geq 0}$ converges to the `fractional kinetic motion' with parameter $\al$, that is a Brownian motion time-changed by the inverse of an independent stable subordinator with index $\alpha$.  On the other hand, in dimension $d=1$ the scaling limit is very different. In fact, for $\alpha\in (0,1)$,  $ ( n^{-1} X_{n^{1+1/\alpha}t})_{t\geq 0}$ converges to 
in distribution, under the annealed law, to the FIN diffusion (for Fontes-Isopi-Newman, see \cite[Definition~1.2]{FIN02}) with parameter $\alpha$, see \cite[Appendix~A]{Ce11}. In any dimension $d\geq 1$, the same processes appear as scaling limits for Bouchaud trap models; we refer to \cite{BC05,Ce06, BACM06, BAC07, BAC08, BCCR15}, also for connections with aging phenomena. For one-dimensional Bouchaud trap models with even heavier tails at infinity,  the trapping mechanism becomes relevant even in the large scale behaviour of the system, so that the model exhibits localization \cite{CM17} and the scaling limit is a spatially-subordinated Brownian motion whose associated clock process is an extremal process \cite{CM15}. Similar results are expected for the RCM.
For one-dimensional RCMs with arbitrarily small conductances having a heavy-tailed distribution at $0$, as well as for the closely related one-dimensional Mott random walk, trapping effects due to the presence of `walls', only occuring in $d=1$, have been studied recently in \cite{CFJ23, CFJ23a, CKS25}; see \cite[Remark~1.9]{CFJ23} for a brief discussion on the connection between those models.

\medskip

The random conductance model has been also studied over other base-graphs than just $\bbZ^d$ or percolation clusters, respectively.  For instance, QFCLTs for random walks under uniformly elliptic random i.i.d.\ conductances on domains with boundary were established in \cite{CCK15}.
Random walks over various point processes in  $\bbR^d$ such as Mott's variable range hopping have been proposed in the physics literature as an effective model describing the phonon-assisted electron transport in disordered solids in the regime of strong Anderson localization. Typically, the conductances of such effective models depend on the temperature, and a main objective is to understand the temperature dependence of the effective conductivity given in terms of the covariance matrix of the limiting Brownian motion. A QFCLT has been shown in \cite{Rou15} for simple random walks  on Delaunay triangulations (the dual of Voronoi tesselations)  generated by point processes in $\mathbb{R}^d$, cf.\ also \cite{Fa23, BG24}, and for certain Mott's variable range hopping \cite{CFP13}.
The methods of Kipnis and Varadhan can be applied even to some deterministic quasiperiodic structures; see e.g.\  \cite{Te10}  for an annealed invariance principle for the simple random walk on Penrose tilings.
Finally, for scaling limits for a class on RCMs on a class of  fractal graphs, see the survey \cite{Cr21} and references therein.

\subsection{Heat kernel behaviour}
 It is well known that in the uniformly elliptic case when the conductances are bounded and bounded away from $0$,  there exist constants $c_1,\ldots, c_4\in (0,\infty)$ such that for all $x,y \in \bbZ^d$ and all $t \geq |x-y|$, 
\begin{align} \label{eq:GB}
  c_1 t^{-d/2} \exp(-c_2|x - y|^2/t) \; \leq \; p^\om(t,x,y) \; \leq \;  c_3 t^{-d/2} \exp(-c_4|x - y|^2/t), 
\end{align}
and similar bounds hold for the heat kernel $q^\om(t,x,y)$ of the CSRW, see \cite{De99}. It is also shown there that in the uniformly elliptic setting such heat kernel bounds are equivalent to a parabolic Harnack inequality.

For random walks on supercritical i.i.d.\ percolation cluster, that is the conductances only take two values $0$ or $1$ with $\prob[\om(e)=1]>p_c$,  for $\prob$-almost all realizations of the conductances, Barlow \cite{Ba04} obtained detailed two-sided Gaussian heat kernel bounds of the form \eqref{eq:GB} for all $x, y$  on the infinite
 cluster and for sufficiently large $t$, extending the on-diagonal heat kernel upper bound shown in \cite{MR04}. Then,  a 
 parabolic Harnack inequality has been shown in \cite{BH09}, leading to a H\"older regularity estimate for space-time harmonic functions, which in combination with the QFCLT allows to derive a  local limit theorem.
 Similar results were obtained in \cite{BD10} for the VSRW under i.i.d.\ conductances uniformly bounded away from zero.  Stronger quantitative homogenization results for heat kernels and Green functions with optimal rates of convergence have been obtained in \cite{DG21} for the VSRW among uniformly elliptic i.i.d.\  conductances on supercritical percolation clusters by using renormalization techniques from quantitative stochastic homogenization, see \cite{AD18} and cf.\  \cite[Chapters 8--9]{AKM19} for details in the case of uniformly elliptic conductances on $\bbZ^d$.

However, due to trapping effects, for i.i.d.\ conductances with values in $[0,1]$ having a sufficiently heavy power-law tail near zero, anomalous heat kernel decay occurs. This phenomenon  was first observed
for annealed return probabilities in \cite{FM06}. Anomalous decay of the quenched heat kernel was first obtained in dimension $d\geq 5$ in \cite{BBHK08} with later extensions and additions in \cite{Bo10,Bo10a,BB12,Bo18}.
In this situation the walk gets trapped for a majority of its time in a small spatial region formed of traps of of the first kind, see \cite{BLRV13}.  Sharp conditions on the tails of i.i.d.\ conductances at zero for Harnack inequalities and a local limit theorem to hold were derived in \cite{BKM15}.
 More precisely, assume
 \begin{align*}
  \om(e) \in [0,1], \qquad  \prob[\om(e) \leq u] \;= \; u^\gamma (1+o(1)), \quad  u\rightarrow 0.
 \end{align*}
Then, Gaussian-type near-diagonal heat kernel bounds, a parabolic Harnack inequality and a local limit theorem hold for the CSRW if $q > d/(8d - 4)$ and for the VSRW if $\gamma>1/4$, see \cite[Theorem~1.8 and 1.9]{BKM15}.
We refer to \cite[Remark~1.10]{BKM15} for a discussion on the optimality of these conditions for both CSRW and VSRW.

 \medskip

To summarize, the upshot of the above results and derivations is that with the random conductance models we are finding ourselves in a somewhat unusual situation when the path distribution satisfies a non-degenerate functional CLT and yet the heat kernel decays anomalously in certain cases; i.e., we have a \emph{CLT without local CLT}.

\section{Random conductance models in ergodic environments} \label{sec:ergodic}

\subsection{Quenched invariance principle}
As discussed in the last section, despite the presence of traps, the QFCLT holds for general i.i.d.\ environments, see Theorem~\ref{thm:QFCLTiid}. This is mainly due to the fact that the sizes of the traps correspond to the sizes of the holes in a supercritical percolation cluster, for which one has effective (polylogarithmic) bounds in the i.i.d.\ case, cf.\ e.g.\ \cite[Lemma~3.3]{Ma08}. However, this property fails once we relax the i.i.d.\ assumption and consider more general degenerate ergodic environments, that is for environments whose law admits long range dependencies, as we will do from now on. See Figure~\ref{fig:numsim} for some simulations comparing i.i.d.\ environments and ergodic conductances with short-range dependencies. 

\renewcommand{\thesubfigure}{}
\begin{figure} 
	\begin{center}
%\fbox{	
\subfigure[(a)]{ \scalebox{0.35}{   \includegraphics{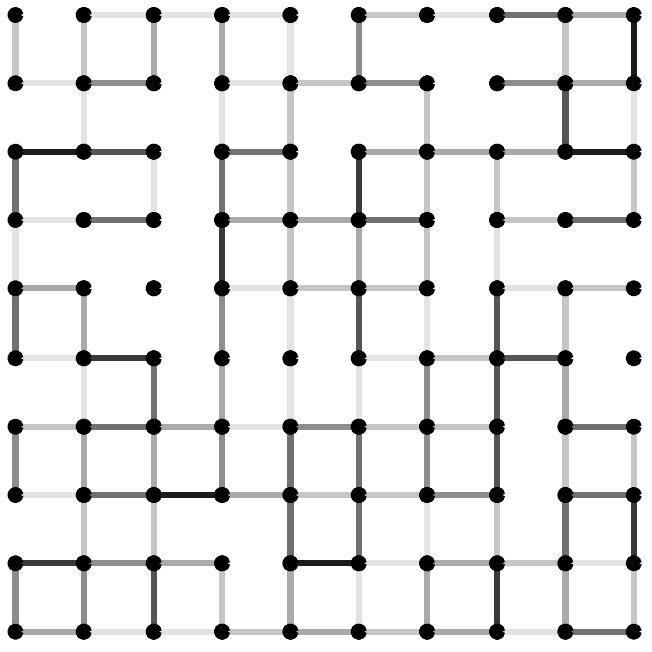} } }
\subfigure[(b)]{ \scalebox{0.35}{   \includegraphics{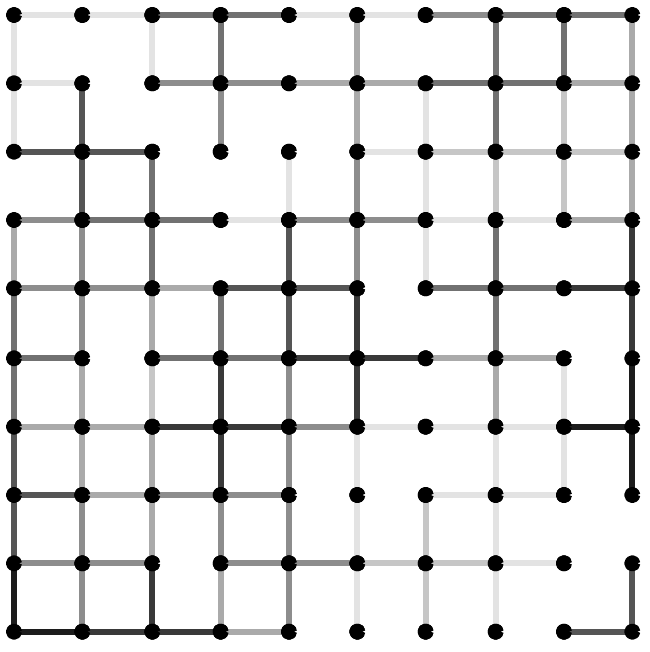} } }
\subfigure[(c)]{ \scalebox{0.35}{   \includegraphics{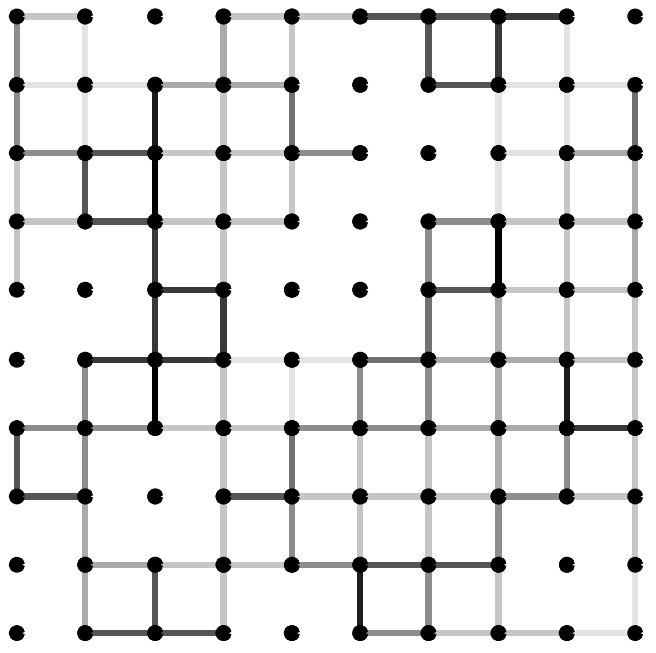} } }	
%}
	\caption[]{Simulations provided by courtesy of Martin Slowik. \\
    (a) $\om(e)\in[0,1]$ i.i.d.\ with  $\mathbb{P}[\om(e) \leq t] \sim t$; \\
  (b) $\om(x,y)=\lambda(x) \vee \lambda(y)$,  $\lambda(x)\in [0,1]$ i.i.d.\ with  $\mathbb{P}[\lambda(x) \leq t] \sim t$;  \\
  (c)  $\om(x,y)=\lambda(x) \wedge \lambda(y)$,  $\lambda(x)\in[0,1]$ i.i.d.\ with  $\mathbb{P}[\lambda(x) \leq t] \sim t$; \\
  Colour code: $\om(e)\in [0,\frac 1 4]$ white (not visible), $\om(e)\in [\frac 1 4, \frac 1 2]$ light grey,  $\om(e)\in [\frac 1 2, \frac 3 4]$ dark grey,  $\om(e)\in [\frac 3 4, 1]$ black.  }
  \end{center}

\label{fig:numsim}
\end{figure}

\begin{assumption} \label{ass:environment}
 Suppose that $\prob$ satisfies the following conditions.
  \begin{enumerate}[ (i)]
    \item $\prob[0<\om(e)<\infty]=1$ and $\mean[\om(e)]<\infty$ for all $e\in E_d$.
    \item 
  The law $\prob$ is stationary and ergodic w.r.t.\ space shifts of $\mathbb{Z}^{d}$, that is, $\prob \circ \tau_{x}^{-1} = \prob$ for all $x \in \mathbb{Z}^{d}$ and $\prob[A] \in \{0, 1\}$ for any $A \in \mathcal{F}$ such that $\tau_{x}^{-1}(A) = A$ for all $x \in \mathbb{Z}^{d}$. 
  \end{enumerate}
\end{assumption}

Assumption~\ref{ass:environment}-(i) is quite natural in a general ergodic setting and it also guarantees non-explosion of the VSRW $X$, cf.\ \cite[Lemma~4.1]{ACDS18} or \cite[Lemma~3.1]{BCKW21},
otherwise explosion could occur in finite time, cf.\ \cite[Remark~6.6]{BD10}. However, additional upper and lower moment conditions are required to ensure the QFCLT to hold,  which can be expressed in terms of $p,q \in [1,\infty]$ as follows.

\begin{definition} \label{def:Mpq}
For any $p,q \in [1,\infty]$ we say that the moment condition $M(p,q)$ holds if 
\begin{align*}
  \mean\!\big[\om(e)^p\big] \; < \; \infty,
  \qquad
 \mean\!\big[\om(e)^{-q}\big] \; < \; \infty, \quad  \text{for any }e \in E_d.
\end{align*}
\end{definition}

In fact, in \cite{BBT16, BBT15} Barlow, Burdzy and Timar gave an example on $\bbZ^2$ for which the QFCLT fails.  In their model, $M(p,q)$ is assumed for $p, q \in (0,1)$ but $\mean[\om(e)] = \infty$ and $\mean[1/\om(e)] = \infty$.  In \cite{Bi11}, Biskup proved the  QFCLT in dimensions $d = 1,2$ under the moment condition $M(p,q)$ with $p = q = 1$.  It is believed that this is the optimal condition for general ergodic environments. Nevertheless, the proof relies on arguments -- inspired by \cite{BB07} for the percolation case -- which cannot be extended to higher dimensions (cf.\ Remark~\ref{rem:sublinear} below). However, in dimensions $d\geq 3$ the QFCLT has been shown to hold under a stronger moment condition.

\begin{theorem}[QFCLT]\label{thm:ip_ergodic}
Let $d \geq 2$. Suppose Assumption~\ref{ass:environment} holds and there exist  $p,q \in (1, \infty]$ with $$ \frac  1  p +  \frac 1 q < \frac 2 {d-1}$$ such that condition $M(p,q)$ is satisfied. 
 Let $X^{(n)} =(n^{-1} X_{n^2t} , t \geq  0)$, $n \in \bbN$. Then, for $\prob$-a.e.\ $\om$, under $\Prob_0^\om$, $X^{(n)}$ converges in
law to a Brownian motion on $\bbR^d$ with a deterministic nondegenerate covariance matrix $\Si_X^2$.
\end{theorem}
\begin{proof}
  For $d \geq 3$, see \cite[Theorem~2]{BS20}, and, for $d = 2$, this can be found in \cite{Bi11}. 
\end{proof}
In dimension $d\geq 2$, Theorem~\ref{thm:ip_ergodic} has first been shown in \cite{ADS15} under the slightly stronger moment condition $M(p,q)$ for $p,q$ satisfying $1/p+1/q<2/d$. 
Recently, Biskup \cite{Bi23} initiated the study of homogenization of random walks under deterministic conductances, where an FCLT is shown for random walks in all environments contained in  a set of deterministic conductance configurations carrying all ergodic conductance laws subject to the moment condition in \cite{ADS15}.

\begin{remark} \label{rem:qip}
   (i) If the law $\prob$ of the conductances is also invariant under symmetries of $\bbZ^d$, then the limiting covariance matrix $\Si_X^2$ must be invariant under symmetries as well, so $\Si_X^2$ is of the form $\Si_X^2 = \si_X^2 I$ for some $\sigma_X>0$, cf.\ Theorem~\ref{thm:QFCLTiid}.

  (ii) Given a speed measure $\pi^\om\!: \bbZ^d \to (0,\infty)$ satisfying $\pi^\om(x) = \pi^{\tau_x \om}(0)$ and $0 < \mean[\pi^{\om}(0)] < \infty$, one can also consider the process, $Y = (Y_t\!: t \geq 0)$ on $\bbZ^d$ that is defined by a time change of $X$, that is, $Y_t \ldef X_{a_t}$ for $t \geq 0$, where $a_t \ldef \inf\{s \geq 0 : A_s > t\}$ denotes the right continuous inverse of the functional
  \begin{align*}
    A_t \;=\; \int_0^t \pi^\om(X_s) \, \md s, \qquad t \geq 0.
  \end{align*}
  Its generator is given by
  \begin{align*}
    \cL_Y^{\om} f(x)
    \;=\;
    \sum_{y \in \bbZ^d} \frac{\om(x,y)}{\pi^{\om}(x)} \big(f(y)-f(x)\big).
  \end{align*}
  Then, if $X$ satisfies a QFCLT, the process $Y$ satisfies a QFCLT as well, see \cite[Section 6.2]{ABDH13}.  In this case, the covariance matrix of its limiting Brownian motion is given by $\Si_Y^2 = \mean[\pi^{\om}(0)]^{-1} \Si_X^2$.  A natural choice for the speed measure is $\pi^{\om} = \mu^{\om}$, in which case $Y$ becomes the CSRW.  

  (iii) In \cite{PRS16}, a QFCLT has been proven for simple random walks on a general class of percolation models with long-range correlations on $\bbZ^d$, $d\geq 2$, introduced in \cite{DRS14}, including the level sets of discrete Gaussian free fields or random interlacements and their vacant set. For random conductance models on this class of random graphs, a QFCLT has been shown in \cite{DNS18} for general ergodic conductances satisfying a moment condition similar to $M(p,q)$ with $p,q\geq 1$ such that $1/p+1/q<2/d$.
\end{remark}

\subsection{A detailed overview of the proof} \label{sec:proofCLT}
In this section we will provide a detailed outline of the proof of Theorem~\ref{thm:ip_ergodic} following in parts the strategy in \cite{ADS15,ACDS18, BS20}, based on a combination of the corrector method and Moser iteration schemes. 

First, we recall that the Markov process $X$ on $\bbZ^d$ naturally induces a corresponding Markov process $(\tau_{X_t} \om : t \geq 0)$ on the space of random environments to which we also refer to as the process of the \emph{environment as seen from the particle}.  By \cite[Lemma~2.4]{ADS15} (see also \cite[Lemma~1.3]{DKS23}), this Markov process is stationary and ergodic with respect to the environment measure $\prob$ on $(\Om, \cF)$.

  The proof of the invariance principle is based on the well-established corrector approach due to Kozlov \cite{Ko85}. More precisely, we introduce harmonic coordinates and construct a corrector $\chi\!: \Om \times \bbZ^d \to \bbR^d$ such that $\Phi(\om, x) = x - \chi(\om, x)$ is an $\cL_X^\om$-harmonic function, that is, for $\prob$-a.e.\ $\om$ and every $x \in \bbZ^d$,
\begin{align} \label{eq:Phi}
  \bigl(\cL_X^{\om} \Phi(\om, \cdot)\bigr)(x)
  \;=\;  
  \sum_{y \sim x} \om(x,y)\, \bigl(\Phi(\om, y) - \Phi(\om, x) \bigr)
  \;=\;
  0.
\end{align}
Moreover, the corrector $\chi$ needs to be shift invariant in the sense that it satisfies a cocycle property.

\begin{definition}
  A measurable function, also called random field, $\Psi\!: \Om \times \bbZ^d \to \bbR$ satisfies the \emph{cocycle property} if for $\prob$-a.e.\ $\om$, 
  \begin{align} \label{eq:cocycle} 
    \Psi(\tau_x \om,y-x)
    \;=\;
    \Psi(\om,y) \,-\, \Psi(\om,x),
    \qquad \text{for all } x, y \in \bbZ^d.
  \end{align}
  We denote by $L^2_{\mathrm{cov}} $ the set of functions $\Psi\!: \Om \times \bbZ^d \to \bbR$ satisfying the cocycle property such that
  \begin{align*}
    \Norm{\Psi}{L^2_{\mathrm{cov}}}^2
    \;\ldef\;
    \mean\!\Big[{\textstyle \sum_{x \sim 0 }}\; \om(0,x) \, \Psi^2(\om, x) \Big] 
    \;<\; \infty.
  \end{align*}
\end{definition}
It can easily be checked that $L_{\mathrm{cov}}^2$ is a Hilbert space. 
 We say a function $\phi\!: \Om \to \bbR$ is \emph{local} if its values only depend on a finite number of edges in a configuration $\om \in \Om$.   We associate to $\phi$ a (horizontal) gradient $\mD \phi\!: \Om \times \bbZ^d \to \bbR$ defined by
\begin{align*}
  \mD \phi (\om,x)
  \;=\;
  \phi(\tau_x \om) - \phi(\om),
  \qquad x \in \bbR^d.
\end{align*}
Obviously, if the function $\phi$ is bounded, $\mD \phi$ is an element of $L_{\mathrm{cov}}^2$. Following \cite{MP07}, we introduce an orthogonal decomposition of the space $L_{\mathrm{cov}}^2$.  Set
\begin{align*}
  L_{\mathrm{pot}}^2
  \;=\;
  \mathop{\mathrm{cl}}
  \big\{ \mD \phi \mid \phi\!: \Om \to \bbR\; \text{ local} \big\}
  \;\text{ in }\;  L_{\mathrm{cov}}^2,
\end{align*}
being the subspace of ''potential'' random fields and let $L_{\mathrm{sol}}^2$ be the orthogonal complement of $L_{\mathrm{pot}}^2$ in $L_{\mathrm{cov}}^2$ called ''solenoidal'' random fields.

In order to define the corrector, we introduce the \emph{position field} $\Pi\!: \Om \times \bbZ^d \to \bbR^d$ with $\Pi(\om,x) = x$.  We  write $\Pi_j$ for the $j$-th coordinate of $\Pi$.  Since $ \Pi_j(\om, y-x) = \Pi_j(\om, y) -\Pi_j(\om, x)$, $\Pi_j$ satisfies the cocycle property.  Moreover,
\begin{align}
  \Norm{\Pi_j}{L_{\mathrm{cov}}^2}^2 
  \;=\;
  \mean \Big[{\textstyle \sum_{x\sim 0}}\; \om(0,x) |x_j|^2\Big]
  \; \leq \;
  \mean[ \mu^\om(0)]
  \;<\;
  \infty,
\end{align}
where $x_j$ denotes the $j$-th coordinate of $x\in \bbZ^d$.  Hence, $\Pi_j \in L_{\mathrm{cov}}^2$.  Thus, we can define $\chi_j \in L_{\mathrm{pot}}^2$ and $\Phi_j \in L_{\mathrm{sol}}^2$ by the property
\begin{align*}
  \Pi_j
  \;=\;
  \chi_j \,+\, \Phi_j
  \;\in\;
  L_{\mathrm{pot}}^2 \oplus L_{\mathrm{sol}}^2.
\end{align*}
This gives a definition of the corrector $\chi = (\chi_1, \dots, \chi_d) : \Om \times \bbZ^d \to \bbR^d$.  We set
\begin{align}\label{eq:def:M}
  M_t \;\ldef \; \Phi(\om, X_t) \;=\; X_t - \chi(\om, X_t).
\end{align}
The following proposition summarizes the properties of the corrector.
\begin{prop}\label{prop:Phi+chi}
  Let $d \geq 2$ and suppose that Assumptions~\ref{ass:environment} holds.  Then, for $\prob$-a.e.\ $\om$, there exist a function $\chi\!: \Om \times \bbZ^d \to \bbR^d$, called the \emph{corrector}, which is characterised by the following properties.
  \begin{enumerate}[(i)]
    \item \emph{Normalisation.} For $\bbP$-a.e.\ $\om$, $\chi(\om, 0) = 0$.

    \item \emph{Stationarity of increments under space shifts.} For $\bbP$-a.e.\ $\om$, $\chi$ satisfies the cocycle property. 
    \item \emph{Weighted square-integrability.} $\chi_j \in   L_{\mathrm{pot}}^2$ for every $j=1,\ldots d$. In particular, $\mean\Bigl[ \sum_{x \sim 0} \om(0,x)\, \bigr| \chi(\om, x) \bigr|^2 \Bigr] < \infty$.

    \item \emph{Harmonic coordinates.} The function $\Phi(\om, x) \ldef x - \chi(\om, x)$ is harmonic in the sense that, for $\prob$-a.e.\ $\om$, $\bigl(\cL_X^{\om}\Phi(\om, \cdot )\bigr)(x)=0$ for all $x \in \bbZ^d$.  In particular, the stochastic process, $(M_t : t \geq 0)$, defined by \eqref{eq:def:M}, is a $\Prob_{x_0}^{\om}$-martingale for any $x_0\in \bbZ^d$.
  \end{enumerate}
\end{prop}
\begin{proof}
  This is widely considered as classical material in homogenisation theory for the RCM.  We refer to, e.g., \cite{Bi11,ABDH13, BD10} for a more detailed exposition and proofs.
\end{proof}

The $\cL_X^{\om}$-harmonicity of $\Phi$ implies that $  M_t = X_t - \chi(\om, X_t)$ is a martingale under $\Prob_{0}^{\om}$ for $\prob$-a.e.\ $\om$, and a QFCLT for the martingale part $M$ can be easily shown by standard arguments, see Proposition~\ref{prop:mconv} below.  We thus get a QFCLT for $X$ once we verify that, for any $T>0$ and $\prob$-a.e.\ $\om$,
\begin{align} \label{eq:conv_corrProb}
  \sup_{t\in [0,T]}\frac{1}{n}\, \bigl| \chi(\om, X_{n^2 t}) \bigr| 
  \underset{n \to 0}{\;\longrightarrow\;} 
  0 
  \qquad \text{in } \Prob_{0}^{\om}\text{-probability}. 
\end{align}

We start by establishing the invariance principle for the martingale part.
\begin{prop}\label{prop:mconv}
  Suppose Assumption \ref{ass:environment} holds, and assume that $\mean\bigl[ 1/\om(e) \bigr] < \infty$ for any $e\in E_d$.  For $n\geq 1$, we write $M^{(n)} = \bigl(\frac{1}{n} M_{n^2 t} : t \geq 0\bigr)$ to denote the diffusively rescaled martingale.  Then, $\prob$-a.s., $M^{n}$ converges in law in the Skorohod topology to a Brownian motion with a non-degenerate covariance matrix $\Si^2_X$ given by
  \begin{align*}
    \Si_{X}^2(i, j)
    \;=\; 
    \mean
    \Bigl[
      {\textstyle \sum_{x \in \bbZ^d}}\, \om(0, x)\, 
      \bigl( e_i \cdot \Phi(\om, x) \bigr)\, 
      \bigl( e_j \cdot \Phi(\om, x) \bigr)
    \Bigr].
  \end{align*}
\end{prop}
\begin{proof}
  The proof is based on the martingale convergence theorem by Helland \cite{He82}; see e.g.\ \cite{ABDH13} or \cite{MP07}.  However, since in \cite{ABDH13, MP07} the actual application of the result in \cite{He82} is slightly inaccurate we give details here.
  
  First, by the Cramer-Wold theorem, it suffices to show that, for every $v \in \bbR^d$ and $\prob$-a.e.\ $\om$, the diffusively rescaled stochastic processes, $v \cdot M^{(n)} = (v \cdot M_{t}^{(n)} : t \geq 0)$, converges as $n$ tends to infinity weakly to a real-valued Brownian motion with variance  $v \cdot \Si_X^2 v$.  Notice that, for any $v \in \bbR^d$, $n \in \bbN$ and $\prob$-a.e.\ $\om$, $v \cdot M^{(n)}$ is a $\Prob_{0}^{\om}$-martingale, and its predictable quadratic variation process $(\langle v \cdot M^{(n)} \rangle_t : t \geq 0)$ is given by
  \begin{align}\label{eq:M_quad_var}
    \bigl\langle v \cdot M^{(n)} \bigr\rangle_t
    \;=\;
    \frac{1}{n^2}
    \int_{0}^{t n^2} 
      \sum_{y \in \bbZ^d} \om(0, y)
      \bigl(v \cdot \Phi(\om, y)\bigr)^2 \circ \tau_{X_s}\,
    \md s,
  \end{align} 
  cf.\ e.g.\ \cite[Proposition~5.11]{ABDH13}.  In order to establish a quenched invariance principle for $v \cdot M^{(n)}$ consider, for any $\de > 0$ and $n \geq 1$, the stochastic process $(N_{t}^{n, \de} : t \geq 0)$  defined by
  \begin{align*}
     N_{t}^{\de, n} 
    &\;\ldef\; 
    \sum_{0 \leq s \leq t} 
    \Bigl( v \cdot M_{s}^{n} - v \cdot M_{s\mm{-}}^{n} \Bigr)^{\!2}\, 
    \indicator_{
      {\{ |v \cdot M_{s}^{n}} - v \cdot M_{s-}^{n}|\geq \de \}
    },
  \end{align*}
  and denote by $( A_{t}^{\de, n} : t \geq 0)$ its compensator.  The latter is defined as the unique predictable process such that $ N^{\de, n} -  A^{\de, n}$ is a local $\Prob_{0}^{\om}$-martingale for $\prob$-a.e.\ $\om$.  For the purpose of determining $A^{\de, n}$, recall that the L\'{e}vy system theorem, see e.g.\ \cite[Theorem VI.28.1]{RW00}, states that, for any function $f\!: \bbZ^{d} \times \bbZ^{d} \to \bbR$ that vanishes on the diagonal, the stochastic process
  \begin{align*}
    \sum_{0 \leq s \leq t} f( X_{s\mm{-}}, X_{s})
    \,-\,
    \int_{(0, t]}
      \sum_{y \in \bbZ^{d}}
      \om(X_{s\mm{-}}, y)
      f(X_{s\mm{-}}, y)\,
    \md s
  \end{align*}
  is a local $\Prob_{0}^{\om}$-martingale for $\prob$-a.e.\ $\om$.  Recalling that $v\cdot \Phi$ satisfies the cocycle property (cf.\ \eqref{eq:cocycle}), we choose $f_{\de, n}(x, y) \ldef \bigl( v \cdot \Phi(\tau_{x} \om, y-x) \bigr)^{2} \indicator_{\{ | v \cdot \Phi(\tau_{x} \om, y-x)| > \de n \}}$ to find that the compensator of the stochastic process $N^{\de, n}$ is given by
  \begin{align}
    A_{t}^{\de, n}
    &\;=\;
    \frac{1}{n^{2}}
    \int_{(0, t n^{2}]}
      \sum_{y \in \bbZ^{d}}
      \om(X_{s\mm{-}}, y)
      f_{\de, n}(X_{s\mm{-}}, y)\,
    \md s
    \nonumber\\[.5ex]
    &\;=\;
    \frac{1}{n^{2}}
    \int_{0}^{t n^{2}}
      \sum_{y \in \bbZ^{d}}
      \om(0, y)
      \bigl( v \cdot \Phi(\om, y)\bigr)^{2}\,
      \indicator_{\{ |v \cdot \Phi(\om, y)| > \de n \}}
      \circ \tau_{X_{s}}\,
    \md s.
    \label{eq:compensator}  
  \end{align}

  Now, by Helland's martingale convergence theorem \cite[Theorem~5.1-(a)]{He82}, see also Brown \cite{Br71}, the sequence of processes, $(v \cdot M^{(n)})_{n \in \bbN}$, converges in law in the Skorohod topology to a Brownian motion with variance  $v\cdot \Sigma^2_X v$ if the following two conditions hold.  For any $t > 0$, $\de > 0$ and $\prob$-a.e.\ $\om$,
  \begin{enumerate}[(i)]
    \item $\langle v \cdot M^{(n)} \rangle_t$ converges to $t \bigl( v \cdot \Si_{X}^{2} v \bigr)$ in $\Prob_{0}^{\om}$-probability as $n \to \infty$.

    \item $A_t^{\de, n}$ converges to zero in $\Prob_{0}^{\om}$-probability as $n \to \infty$.
  \end{enumerate}
 
  For (i), in view of \eqref{eq:M_quad_var}, an application of the (time) ergodic theorem, applied on the environment process $( \tau_{X_{t}} \om : t \geq 0)$, yields that, $\prob \times \Prob_{0}^{\om}$-a.s.,
  \begin{align*}
    \lim_{n \to \infty} \langle v \cdot M^{n}\rangle_t
    \;=\;
    t \sum_{y \in \bbZ^{d}}
    \mean\Big[
      \om(0,y) \,
      \bigl(v \cdot \Phi(\om, y) \bigr)^2
    \Big]
    \;=\;
    t\, \bigl( v \cdot \Si_{X}^{2}v \bigr)
  \end{align*}
  with $\Si_{X}^{2}$ as in the statement, so that condition (i) follows.

To show condition (ii), fix some $L \in (0, \infty)$.  Then, by using again the (time) ergodic theorem on the environment process $(\tau_{X_{t}} \om : t \geq 0)$, it follows from the representation in \eqref{eq:compensator} that, $\prob \times \Prob_{0}^{\om}$-a.s.,
  \begin{align*}
    \limsup_{n \to \infty} A_{t}^{\de, n}
    &\;\leq\;
    \limsup_{n \to \infty}
    \frac{1}{n^{2}}
    \int_{0}^{t n^{2}}
       \sum_{y \in \bbZ^{d}}
      \om(0, y)
      \bigl( v \cdot \Phi(\om, y)\bigr)^{2}\,
      \indicator_{\{ |v \cdot \Phi(\om, y)| > L \}}
      \circ \tau_{X_{s}}\,
    \md s
    \\[.5ex]
    &\;=\;
    t \sum_{y \in \bbZ^{d}}
    \mean\Big[
      \om(0,y) \,
      \bigl(v \cdot \Phi(\om, y) \bigr)^2\,
      \indicator_{\{ |v \cdot \Phi(\om, y)| > L \}}
    \Big].
  \end{align*}
  Thus, we obtain that, for any $\de > 0$ and $t \geq 0$,
  \begin{align*}
    \limsup_{L \to \infty} \limsup_{n \to \infty} A_{t}^{\de, n}
    \;=\;
    0,
    \qquad \prob \times \Prob_{0}^{\om}\text{-a.s.}, 
  \end{align*}
  which implies (ii) and completes the proof of the QFCLT for the martingale $M$.

  Finally, we refer to \cite[Proposition~4.1]{Bi11} for a proof that the limiting covariance matrix $\Si_{X}^{2}$ is non-degenerate.
\end{proof}

We now turn to the proof of \eqref{eq:conv_corrProb}. The most common approach to show \eqref{eq:conv_corrProb} and to deduce the QFCLT is to establish the sublinearity of the corrector. 
In a first step we show $\ell^1$-sublinearity of the corrector. Note that this only requires the weaker moment condition $M(1,1)$ to hold. 

From now on, for any $z_0 \in \bbZ^d$ and $r \geq 0$, we denote by $B(z_0, r) \ldef \{y \in \bbZ^d : |z_0-y| \leq r\}$ the closed ball with centre $z_0$ and radius $r$ with respect to the natural graph distance on $\bbZ^d$. For abbreviation, set $B_r\ldef B(0,r)$. Finally, for any $A \subset \bbZ^d$ we write $|A|$ for the number of elements in $A$.

\begin{prop} [$\ell^1$-sublinearity]\label{prop:l1_sublin}
  Suppose Assumption \ref{ass:environment} holds, and assume that $\mean\bigl[ 1/\om(e) \bigr] < \infty$ for any $e\in E_d$. Then
  \begin{align*}
    \lim_{n\to\infty} \frac 1 {|B_n|} \sum_{x\in B_n} \frac {|\chi_j(\om,x)|} n=0, \qquad  \text{$\prob$-a.s.}, \, \forall j\in \{1,\ldots,d\}.
    \end{align*}
\end{prop}
\begin{proof}
We present here a rather concise argument following  \cite[Lemma~2]{BFO18} (cf.\ also \cite[Proposition~2.9]{ADS15} for an alternative proof). 

\smallskip

  \emph{Step~1.} In a first step we show that 
  \begin{align} \label{eq:l1_step1}
    \lim_{n\to\infty} \frac 1 {|B_n|} \sum_{x\in B_n} \frac {|\chi_j(\om,x) - (\chi_j(\om,\cdot))_{B_n}|} n=0, \qquad  \text{$\prob$-a.s.}, \, \forall j\in \{1,\ldots,d\},
    \end{align}
where we write $(u)_{B_n} \ldef \frac{1}{|B_n|} \sum_{y \in B_n} u(y)$ for any function $u\!: \bbZ^d \to \bbR$. Recall that $\chi_j \in   L_{\mathrm{pot}}^2$. Thus, there exist bounded functions $\phi_{j,k}: \Om \to \bbR$ such that $D\phi_{j,k} \rightarrow \chi_j$ in $L_{\mathrm{cov}}^2$ as $k\to \infty$. 
Further, recall that by the local $\ell^1$-Poincar\'{e} inequality on $\bbZ^d$,  there exists a constant $C_{\mathrm{P}} \in (0, \infty)$ such that for every $u\!: \bbZ^d \to \bbR$,
\begin{align*}
  \sum_{x\in B_n}\mspace{-6mu} \big| u(x) - (u)_{B_n} \big|
  \;\leq\;
  C_{\mathrm{P}}\, n\, 
  \sum_{\substack{x, y \in B_n\\ y \sim x}}
 \big| u(x) - u(y)\big|.
\end{align*}
Using this and the fact that $\chi_j$ and $D\phi_{j,k}$ satisfy the cocycle property, we have
\begin{align*}
  & \frac 1 n  \frac 1 {|B_n|} \sum_{x\in B_n} \big|\chi_j(\om,x) - (\chi_j(\om,\cdot))_{B_n}\big| \\
 & \mspace{36mu} \; \leq \; 
  \frac 1 n  \frac 1 {|B_n|} \sum_{x\in B_n} \Big|(\chi_j-D\phi_{j,k})(\om,x) - \big((\chi_j-D\phi_{j,k})(\om,\cdot)\big)_{B_n}\Big| + \frac 4 n \|\phi_{j,k}\|_\infty \\
& \mspace{36mu}
\; \leq \;
C_{\mathrm{P}}\, \frac 1 {|B_n|}
      \sum_{\substack{x, y \in B_n\\ y \sim x}}
      \big|(\chi_j -D\phi_{j,k})(\tau_x \om, y-x) \big|  + \frac 4 n \|\phi_{j,k}\|_\infty.
\end{align*}
Since $\phi_{j,k}$ is bounded the last term converges to zero as $n\to \infty$. Hence, by the ergodic theorem, $\prob$-a.s.,
\begin{align*}
  & \limsup_{n\to \infty} \frac 1 n  \frac 1 {|B_n|} \sum_{x\in B_n} \big|\chi_j(\om,x) - (\chi_j(\om,\cdot))_{B_n}\big|
   \leq  C_{\mathrm{P}}\,\mean\!\Big[ \sum_{y\sim 0} \big|(\chi_j -D\phi_{j,k})(\om, y) \big|  \Big] \\
   & \mspace{36mu}
   \; \leq \;
  \bigg( \sum_{y\sim 0}\mean\!\big[ \om(0,y)^{-1} \big]\bigg)^{\! 1/2} \,  \Norm{\chi_j-D\phi_{j,k}}{L^2_{\mathrm{cov}}},
 \end{align*}
where we used the Cauchy-Schwarz inequality in the last step. Recalling that $\Norm{\chi_j-D\phi_{j,k}}{L^2_{\mathrm{cov}}}\rightarrow 0$ for $k\to \infty$, we obtain \eqref{eq:l1_step1}.
\medskip

\emph{Step~2.} From \eqref{eq:l1_step1} it follows that, for $\prob$-a.e.\ $\om$, for every $\delta>0$ there exists $n_0=n_0(\om,\delta)$ such that, for all $n\geq n_0$,
\begin{align*}
  \frac 1 {|B_n|} \sum_{x\in B_n} \frac {|\chi_j(\om,x) - (\chi_j(\om,\cdot))_{B_n}|} n \; \leq \; \delta.
\end{align*}
Set $n_k\ldef 2^k n_0$. Then, for all $k\geq 1$, using the volume regularity of the balls, 
\begin{align*}
 & \frac 1 {n_k} \Big| \big( \chi_j (\om, \cdot) \big)_{B_{n_k}} - \big(\chi_j (\om, \cdot\big))_{B_{n_{k-1}}} \Big|
   \leq 
  \frac 1 {n_k}  \frac 1 {| B_{n_{k-1}}|} \sum_{x\in B_{n_{k-1}}} \Big| \chi_j(\om,x) - \big( \chi_j (\om, \cdot) \big)_{B_{n_k}} \Big| \\
  & 
  \mspace{36mu}
  \; \leq \; 
  \frac c {n_k}  \frac 1 {| B_{n_{k}}|} \sum_{x\in B_{n_{k}}} \Big| \chi_j(\om,x) - \big( \chi_j (\om, \cdot) \big)_{B_{n_k}} \Big|
  \; \leq \; c \delta.
\end{align*}
Hence, noting that $n_k= 2^{k-l}n_l$ for $1\leq l \leq k$, we get
\begin{align*}
 & \frac 1 {n_k} \Big| \big( \chi_j (\om, \cdot) \big)_{B_{n_k}} - \big(\chi_j (\om, \cdot\big))_{B_{n_{0}}} \Big| 
  \; \leq \;
  \sum_{l=1}^k  \frac 1 {n_k} \Big| \big( \chi_j (\om, \cdot) \big)_{B_{n_l}} - \big(\chi_j (\om, \cdot\big))_{B_{n_{l-1}}} \Big| \\
  & \mspace{36mu}
  \; \leq \;
  c \delta \sum_{l=1}^k \frac{1}{2^{k-l}} 
  \; \leq \;
  2 c \delta.
\end{align*}
Thus, by the triangle inequality, for all $k\geq 1$,
\begin{align*}  
  I(n_k) &\ldef \frac 1 {n_k} \, \frac 1 {|B_{n_k}|} \sum_{x\in B_{n_k}} 
  |\chi_j(\om,x)| \\
  & \leq 
  \delta + \frac 1 {n_k} \Big| \big( \chi_j (\om, \cdot) \big)_{B_{n_k}} - \big(\chi_j (\om, \cdot\big))_{B_{n_{0}}} \Big| + \frac 1 {n_k} \big| \big(\chi_j (\om, \cdot\big))_{B_{n_{0}}} \big| \\
&
\leq
(2c+1) \delta + \frac 1 {n_k} \big| \big(\chi_j (\om, \cdot\big))_{B_{n_{0}}} \big|,
\end{align*}
where the last term converges to zero as $k\to \infty$. Since $\delta>0$ is arbitrary, this implies $\lim_{k\to \infty} I(n_k)=0$. 
Noting that $I(n)\leq c I(n_k)$ if $n_{k-1} < n \leq n_k$, we get the claim.
\end{proof}

It remains to upgrade the $\ell^1$-sublinearity in Proposition~\ref{prop:l1_sublin} to an $\ell^\infty$-sublinearity statement. The proof is based on a maximal inequality for harmonic functions, proved by Moser iteration schemes, and requires the stronger (non-optimal) moment condition.
For any non-empty, finite $B \subset \bbZ^d$ and $p \in (0, \infty)$, we introduce space-averaged norms on functions $f\!: B \to \bbR$ by
\begin{align*}
  \Norm{f}{p,B}
  \;\ldef\;
  \biggl(\frac{1}{|B|}\sum_{x \in B} |f(x)|^p\biggr)^{\!\!1/p}, 
  \qquad
  \Norm{f}{\infty, B}
  \;\ldef\;
  \max_{x \in B} |f(x)|.
\end{align*}
For abbreviation set
\begin{align*}
  \nu^\om(x) \ldef \sum_{y\sim x} \frac 1 {\om(x,y)}, \qquad x\in \bbZ^d.
\end{align*}

\begin{theorem}[Maximal inequality] \label{thm:MI}
  Let $d\geq 3$. For any $x_0 \in \bbZ^d$ and $n \geq 1$, suppose that $\cL^{\om} u = 0$ on $B(x_0,2n)$.   Then, for any $p,q \in (1, \infty]$ with
  \begin{align*}%\label{eq:cond:pq}
    \frac{1}{p} + \frac{1}{q} \;<\; \frac{2}{d-1}, 
  \end{align*}
 we have that 
  \begin{align}\label{eq:MP}
    \max_{x \in B(x_0,n)} |u(x)|
    \;\leq\;
    \Lambda^\om(B(x_0,2n)) 
    \, \Norm{u}{1, B(x_0,2 n)},
  \end{align}
  where, for any $B\subset \bbZ^d$,
\begin{align*}
  \Lambda^\om(B) \; \ldef \; c_5 \, \Big( 1 \vee
  \Norm{\mu^{\om}}{p, B}\, \Norm{\nu^{\om}}{q,B}
\Big)^{\!\ka}
\end{align*}
  for some $\ka = \ka(d,p,q) \in (1,\infty)$ and $c_5 = c_5(d,p,q) \in (0,\infty)$.
\end{theorem}
\begin{proof}
 See \cite[Theorem~2]{BS20}. In dimension $d\geq 2$,  for $p,q \in (1,\infty]$ satisfying the slightly stronger condition $1/p+1/q<2/d$, similar maximal inequalities have been shown in \cite[Proposition~3.2 and Corollary~3.4]{ADS16} on a general class of underlying graphs, cf.\ also the parabolic version for possibly time-dependent conductances in \cite[Theorem~2.7]{ACS21}. All those proofs are based on inductive schemes such as Moser or De~Giorgi iterations. 
\end{proof}

Finally, the following generalisation of the ergodic theorem will help us to control ergodic averages on scaled balls with varying centre-points. 
\begin{prop} \label{prop:krengel_pyke}
  Let $K>0$ and  $\cB \ldef \bigl\{ B : B \text{ closed Euclidean ball in } [-K,K]^d\bigr\}$.  Suppose that Assumption~\ref{ass:P} holds.  Then, for any $f \in L^1(\Om)$,
  \begin{align*}
    \lim_{n \to \infty} \sup_{ B \in \cB}
    \biggl|
      \frac{1}{n^{d}}\, \sum_{x \in (nB) \cap \bbZ^d}\mspace{-12mu} f \circ \tau_{x}
      \,-\,
      |B| \cdot \mean\bigl[ f \bigr]
    \biggr|
    \;=\;
    0, 
    \qquad \prob\text{-a.s.},
  \end{align*}
  where $|B|$ denotes the Lebesgue measure of $B$.
\end{prop}
\begin{proof}
  See, for instance, \cite[Theorem~1]{KP87}.
\end{proof}

For abbreviation write
$\chi^{(n)}(\om, x) \ldef \frac{1}{n}\, \chi(\om, x)$.

\begin{prop} \label{prop:sublin_corr}
  Under the assumptions of Theorem~\ref{thm:ip_ergodic}, for any $L \geq 1$ and $j = 1, \ldots, d$,
  \begin{align} \label{eq:sublin_corr}
    \lim_{n \to \infty} \max_{x \in B(0, L n)} \big| \chi_j^{(n)}(\om, x) \big|
    \;=\;
    0,
    \qquad \prob\text{-a.s.}
  \end{align}
\end{prop}
\begin{proof}
We follow the arguments in \cite[Proposition~2]{BS20}. 
Fix $m \in \bbN$. For $n$ sufficiently large compared to $m$, we cover the ball $B(0,n)$ with finitely many balls $B(z_i, \lfloor \frac n m \rfloor)$,  such that $z_i\ldef \lfloor \frac n m \rfloor y_i$ for some $y_i \in [-2m, 2m]^d\cap \bbZ^d$, $i=1,\ldots, k_{m,n}$.
  For every $j = 1, \ldots, d$ and $z \in \bbZ^d$ define
  \begin{align*}
    u_j^z(\om, x)\ldef\chi_j(\om,x) -e_j \cdot (x-z)= - \Phi(\om,x)+ e_j\cdot z,
  \end{align*} 
 where $e_j$ denotes the $j$-th unit vector in $\bbZ^d$.
  Obviously, $u^z_j(\om, \cdot)$ is $\cL^\om_X$-harmonic. Thus, it follows from Theorem~\ref{thm:MI} that, for any $z\in \bbZ^d$, 
  \begin{align*} 
    \max_{x \in B(z,\lfloor \frac n m \rfloor)} \big| u_j^{z}(\om, x) \big|
    \;\leq\;
    \Lambda^\om(B(z, 2\lfloor \tfrac n m \rfloor))
    \Norm{u_j^{z}(\om, \cdot)}{1, B(z, 2 \lfloor \frac n m \rfloor)}.
  \end{align*}

Since $|e_j \cdot (x-z)|\leq c n/m$ for all $x\in B(z,\lfloor \frac n m \rfloor)$ and a suitable constant $c>0$, this implies 
\begin{align*}
  \max_{x \in B(z,\lfloor \frac n m \rfloor)} \big| \chi_j(\om, x) \big|
  & \;\leq \;
  \max_{x \in B(z,\lfloor \frac n m \rfloor)} \big| u_j^{z}(\om, x) \big| + c n/m \\
  &  \;\leq\;
    \Lambda^\om(B(z, 2\lfloor \tfrac n m \rfloor))
    \Norm{u_j^{z}(\om, \cdot)}{1, B(z, 2 \lfloor \frac n m \rfloor)} + c n/m \\
    &  \;\leq\;
    \Lambda^\om(B(z, 2\lfloor \tfrac n m \rfloor))
   \Big( \Norm{\chi_j(\om, \cdot)}{1, B(z, 2 \lfloor \frac n m \rfloor)} + c n/m \Big)+ c n/m.
  \end{align*}
  Hence,
  \begin{align*}
    & \max_{x \in B(0, n)} \big| \chi_j^{(n)}(\om, x) \big|
    \; \leq \; 
    \max_{i=1,\ldots, k_{m,n}}  \max_{x \in B(z_i,\lfloor \frac n m \rfloor)} \big| \chi_j^{(n)}(\om, x) \big| \\
    & \mspace{36mu} \; \leq \;
    \max_{i=1,\ldots, k_{m,n}} 
    \Lambda^\om(B(z_i, 2\lfloor \tfrac n m \rfloor))
    \Big( \Norm{\chi_j^{(n)}(\om, \cdot)}{1, B(z_i, 2 \lfloor \frac n m \rfloor)} + c /m \Big)+ c /m \\
    & \mspace{36mu} \; \leq \;
    \max_{i=1,\ldots, k_{m,n}} 
    \Lambda^\om(B(z_i, 2\lfloor \tfrac n m \rfloor))
    \Big( c m^d \Norm{\chi_j^{(n)}(\om, \cdot)}{1, B(0, 2n)} + c /m \Big)+ c /m.
  \end{align*}
 Now, since $m\in \bbN$ is fixed, we invoke the version of the ergodic theorem in Proposition~\ref{prop:krengel_pyke} to obtain that
\begin{align*}
 &\limsup_{n\to \infty} \max_{i=1,\ldots, k_{m,n}} \Lambda^\om(B(z_i, 2\lfloor \tfrac n m \rfloor))
 \; \leq \;
 \limsup_{n\to \infty} \max_{y\in [-2m,2m]^d\cap \bbZ^d} \Lambda^\om(B(\lfloor \tfrac n m \rfloor y, 2\lfloor \tfrac n m \rfloor)) \\
& \mspace{36mu} \; \leq \;
\limsup_{n\to \infty} \max_{y\in [-2m,2m]^d\cap \bbZ^d}  c_5  \Big( 1 \vee
\Norm{\mu^{\om}}{p, B(\lfloor \tfrac n m \rfloor y, 2\lfloor \tfrac n m \rfloor)}\, \Norm{\nu^{\om}}{q,B(\lfloor \tfrac n m \rfloor y, 2\lfloor \tfrac n m \rfloor)}
\Big)^{\!\ka} \\
& \mspace{36mu} \; \leq \;
c \Big( 1 \vee \mean\big[\mu^\om(0)^p\big]^{1/p} \mean\big[\nu^\om(0)^q\big]^{1/q} \Big)^{\!\ka} < \infty.
\end{align*}
Combining the above with the $\ell^1$-sublinearity in Proposition~\ref{prop:l1_sublin} we get that
\begin{align*}
  \limsup_{n\to \infty}  \max_{x \in B(0, n)} \big| \chi_j^{(n)}(\om, x) \big|
    \; \lesssim \; 
    m^{-1} \Big( 1 \vee \mean\big[\mu^\om(0)^p\big]^{1/p} \mean\big[\nu^\om(0)^q\big]^{1/q} \Big)^{\!\ka} + m^{-1}.
\end{align*}
Since $m\in \bbN$ is arbitrary, this yields \eqref{eq:sublin_corr}  for $L=1$, and the trivial identity
\begin{align*}
  \lim_{n\to \infty} \frac 1 n \max_{x \in B(0, Ln)} \big| \chi_j(\om, x) \big| = L  \lim_{n\to \infty} \frac 1 n \max_{x \in B(0, n)} \big| \chi_j(\om, x) \big| = 0
\end{align*}
completes the proof.  
\end{proof}

\begin{proof}[Proof of Theorem~\ref{thm:ip_ergodic}]
  We may follow the arguments in \cite[Proposition~2.13]{ADS15} to show that Proposition~\ref{prop:sublin_corr} implies \eqref{eq:conv_corrProb}. Combining this with the QFCLT for the martingale part gives the result.
\end{proof}

\begin{remark} \label{rem:sublinear}
  The above proof uses the rather common route to establish the \emph{sublinearity everywhere} of the corrector in form of Proposition~\ref{prop:sublin_corr} to obtain the QFCLT.
However, the (non-optimal) moment condition $M(p,q)$ for $p,q\geq 1$ satisfying $1/p+1/q<2/(d-1)$ is known to be infinitesimally close to sharp for Proposition~\ref{prop:sublin_corr} to hold. In fact, in $d\geq 3$, for any $p,q\geq 1$ such that $1/p+1/q>2/(d-1)$, there exists an ergodic law $\prob$, satisfying $M(p,q)$, for which the corrector is (well defined yet) not sublinear everywhere, see \cite[Theorem~2.6]{BCKW21}. We refer to \cite{BCKW21, BM15} for approaches to show a QFCLT that avoid proving everywhere-sublinearity of the corrector. 
\end{remark}

\subsection{Local limit theorems}
While the aforementioned moment conditions are non-optimal for the QFCLT, they turn out to be (close to) sharp for the stronger quenched local limit theorem to hold. Recall that we denote by $p^\om(t,x,y)$ the heat kernel of the VSRW $X$ and by $q^\om(t,x,y)$ the heat kernel of the CSRW $Y$. Further, we write $\bar p^{\Sigma_X}(t,x,y)$ and $\bar p^{\Sigma_Y}(t,x,y)$ for the Gaussian transition densities with covariance matrices $\Sigma^2_X$ and $\Sigma^2_Y$, respectively, appearing in the QFCLT for $X$ and $Y$, see Theorem~\ref{thm:ip_ergodic}, Remark~\ref{rem:qip} and \eqref{eq:GaussHK} above. 

\begin{theorem}[Quenched local CLT]
  Let $d \geq 2$ and suppose Assumption~\ref{ass:environment} holds.

  \begin{enumerate}[(i)]
    \item \emph{Variable speed random walk.} Assume there exist  $p \in (1, \infty]$  and $q \in (d/2, \infty]$  with $$ \frac  1  p +  \frac 1 q < \frac 2 {d-1}$$ such that condition $M(p,q)$ is satisfied. Then, for any given compact sets $I \subset (0, \infty)$ and $K \subset \bbR^d$,
    \begin{align*}
      \lim_{n\to \infty} \sup_{x\in K} \sup_{t\in I} \big| n^d p^\om(n^2 t, 0, \lfloor nx \rfloor) - \bar p^{\Sigma_X}(t,0,x) \big| =0, \qquad \text{$\prob$-a.s.}
    \end{align*}

    \item \emph{Constant speed random walk.} Assume there exist  $p,q \in (1, \infty]$ with $$ \frac  1  p +  \frac 1 q < \frac 2 d$$ such that condition $M(p,q)$ is satisfied. Then, for any given compact sets $I \subset (0, \infty)$ and $K \subset \bbR^d$,
    \begin{align*}
      \lim_{n\to \infty} \sup_{x\in K} \sup_{t\in I} \big| n^d q^\om(n^2 t, 0, \lfloor nx \rfloor) - a \, \bar p^{\Sigma_Y}(t,0,x) \big| =0, \qquad \text{$\prob$-a.s.},
    \end{align*}
    with $a\ldef 1/\mean[\mu^\om(0)]$. 
  \end{enumerate}
\end{theorem}
\begin{proof}
  See \cite[Theorem~4]{BS22} for (i) and \cite[Theorem~1.11]{ADS16} for (ii).
\end{proof}
\begin{remark}
 The moment condition in (ii) is infinitesimally close to sharp for the quenched local limit theorem to hold for the CSRW. In fact, for any $p,q\geq 1$ with $1/p+1/q>2/d$, $d\geq 2$, there exists an ergodic law $\prob$ satisfying condition $M(p,q)$, that generates sufficiently many deep traps so that the quenched local limit theorem fails, see \cite[Theorem~5.4]{ADS16}. 

In a similar spirit, for the VSRW, Deuschel and Fukushima \cite{DF20} constructed a stationary and ergodic environment, built from random layered conductances with $\om(e) \geq 1$ for $\prob$-a.e.\  $\om$ (that is
$q = \infty)$ and $\mean[\om(e)^p] < \infty$ with $p < (d-1)/
2$, such that the quenched local limit theorem
fails in dimension $d \geq 4$, see \cite[Proposition~1.5]{DF20}.
\end{remark}

Unlike the QFCLT, the quenched local limit theorem is not easily preserved under time-changes, cf.\ Remark~\ref{rem:qip}-(ii) above. Nevertheless, for RCMs under a class of stationary speed measures a quenched local limit theorem has been derived in \cite{AT21}, again under suitable moment conditions.  In random graph setting, mentioned in Remark~\ref{rem:qip}-(iii) above, a quenched local limit theorem has been established for the VSRW in \cite[Section~5]{ACS21} under as the same conditions as for the QFCLT in \cite{DNS18}. Similar homogenization results have been obtained for a class of symmetric diffusions in degenerate random media in \cite{CD16, CD15}.

The main strategy to prove a local limit theorem follows the idea from \cite{BH09}. The proof requires the QFCLT and  in addition some H\"older regularity estimate on the macroscopic scale for the heat kernel, which can be directly deduced from an oscillation inequality for space-time harmonic functions. In \cite{ADS16}, the latter is obtained from a parabolic Harnack inequality, see \cite[Theorem~1.4]{ADS16}. However, for the derivation of the required oscillation inequality a weaker version of a parabolic Harnack inequality is sufficient, as observed e.g.\ in \cite{ACS21, BS22, CKW24}. 
The proofs rely on extensions of classical Moser and De~Giorgi iteration techniques to discrete finite-difference divergence-form operators with degenerate coefficients.

\subsection{Heat kernel estimates}
Next we discuss heat kernel estimates in the general ergodic setting. We begin with the following Gaussian heat kernel upper bound for the CSRW. 

\begin{theorem}[\cite{ADS16a}] \label{thm:hkeCSRW}
 Let $d\geq 2$ and suppose that Assumption~\ref{ass:environment} holds. Assume there exist  $p,q \in (1, \infty]$ with $$ \frac  1  p +  \frac 1 q < \frac 2 d$$ such that condition $M(p,q)$ is satisfied.
 Then, there exist constants $c_i \in (0,\infty)$, $i=6,\ldots,9$, only depending on $d, p, q, \mean[\om(e)^p]$ and $\mean[\om(e)^{-q}]$, such that, for $\prob$-a.e.\ $\om$,  for all $x \in \bbZ^d$ there exists $N_1(x)=N_1(x, \om) < \infty$ such that for any given $t$ with $\sqrt{t} \geq N_1(x)$ and all $y \in \bbZ^d$ the following hold.
  \begin{enumerate}
    \item [(i)] If $|x-y|\leq c_6 t$ then
      \begin{align*}
        q^\om(t,x, y)
        \;\leq\;
        c_7\, t^{-d/2}\,  \exp\!\big(-c_8 |x-y|^2/t\big).
      \end{align*}
    \item [(ii)] If $|x-y|\geq c_6 t$ then
      \begin{align*}
        q^\om(t,x, y)
        \;\leq\;
        c_7 \, \exp\!\big(-c_9 |x-y| (1 \vee \log(|x-y|/t))\big).
      \end{align*}
  \end{enumerate}
\end{theorem}

Here the random constants $N_1(x,\om)$, $x\in \bbZ^d$, represent the minimal scale such that the  ergodic averages $\|\mu^{\om}\|_{p,B(x,n)}$ and $\|\nu^{\om}\|_{q,B(x,n)}$ of conductances, taken over balls centred at $x$, can be controlled in terms of a deterministic quantity for $n\geq N_1(x,\om)$. Nevertheless, for general ergodic environments we cannot control the size of the random variable $N_1(x)$, $x\in \bbZ^d$, as
this requires some information on the speed of convergence in the ergodic theorem. However, if we additionally assume, for instance, that the environment satisfies a concentration inequality in form of a spectral gap inequality w.r.t.\ the so-called vertical derivative, then $\mean[N_1(x)^n] < \infty$ provided a stronger moment condition holds (depending on n), see \cite[Assumption~1.3 and Lemma~2.10]{AN17}. See also \cite{AH21} for some alternative conditions.

The proof of Theorem~\ref{thm:hkeCSRW} (and its extension in Theorem~\ref{thm:hkeVSRW} below) is based on a combination of Moser iteration technique with Davies' perturbation method, see e.g.\ \cite{Da89, Da93, CKS87}.

Clearly, it would be desirable to establish matching lower bounds.  It is well known that Gaussian lower and upper bounds on the heat kernel are equivalent to a parabolic Harnack inequality in many situations, for instance in the case of uniformly elliptic conductances, see \cite{De99}.   As mentioned above, under the assumptions of Theorem~\ref{thm:hkeCSRW} a parabolic Harnack inequality has been obtained in \cite{ADS16}.  However, due to the special structure of the Harnack constant in \cite{ADS16}, in particular its dependence on $\|\mu^{\om}\|_{p,B(x,n)}$ and $\|\nu^{\om}\|_{q,B(x,n)}$, one cannot directly deduce off-diagonal Gaussian lower bounds from it.  More precisely, in order to get effective Gaussian off-diagonal bounds using the established chaining argument (see e.g.\ \cite{Ba04}), one needs to apply the Harnack inequality on a number of balls with radius $n$ having a distance of order $n^2$.  In general, the ergodic theorem does not 
give the required uniform control on the convergence of space-averages of stationary random variables over such balls (see \cite{AJ75}). This can be resolved by imposing stronger moment conditions and additional mixing assumptions, under which a matching Gaussian lower bound has been shown in \cite{AH21}.

\medskip

Next we discuss upper bounds on the heat kernel $p^{\om}(t, x ,y)$ of the VSRW. For that purpose we need to introduce the distance $d_\om$ defined by
\begin{align} \label{eq:d_om}
  d_{\om}(x,y)
  \;\ldef\;
  \inf_{\gamma}
  \Bigg\{
    \sum_{i=0}^{l_{\gamma}-1} 1 \wedge \om(z_i,z_{i+1})^{-1/2}
  \Bigg\},
\end{align}
where the infimum is taken over all paths $\ga = (z_0, \ldots, z_{l_\gamma})$  connecting $x$ and $y$.  Note that $d_{\om}$ is a first passage percolation metric which is adapted to the transition rates of the random walk. In general, $d_{\om}$ can be identified with the intrinsic metric generated by the Dirichlet form associated with $\cL_{X}^{\om}$ and $X$, cf.\ \cite[Proposition~2.3]{ADS19}.
In particular, $d_\om$ is the metric that is expected to govern the off-diagonal heat kernel decay, cf.\ \cite{Da93a, Da93}. Further, notice that $d_\om$ is obviously bounded from above by the graph distance on $\bbZ^d$. In fact, $d_\om$ can become much smaller than the graph distance.

\begin{theorem} \label{thm:dist}
 Let $d\geq 2$ and suppose that Assumption~\ref{ass:environment} holds. Assume that there exists $p > (d-1)/2$ such that $\mean[\om(e)^p]<\infty$.
  Then, there exists $c_{10} > 0$, only depending on $\mean[\om(e)^p]$ such that the following holds. For every $x \in \bbZ^d$ and $\prob$-a.e.\ $\om$,  there exists $N_2(\om, x) < \infty$ such that  for any $y \in \bbZ^d$ with $|x-y| \geq N_2(\om, x)$,
  \begin{align} \label{eq:LBdist}
    d_\om(x,y)
    \;\geq\;
    c_{10}\, |x-y|^{1-\frac{d-1}{2p}}.
  \end{align}
\end{theorem}
\begin{proof}
  See \cite[Theorem~2.4]{ADS19}.
\end{proof}

Despite the fact that Theorem~\ref{thm:dist} follows by a relatively simple application of H\"older's inequality, this lower bound turns out to be at least close to optimal within a general ergodic framework. In fact, there exists an ergodic environment built from layered conductances being constant along lines but independent between different lines, for which the lower bound in \eqref{eq:LBdist} is attained up to an arbitrarily small correction in the exponent, see  \cite[Theorem~2.5]{ADS19}.

\begin{theorem}[\cite{ADS19}] \label{thm:hkeVSRW}
  Let $d\geq 2$ and suppose that Assumption~\ref{ass:environment} holds. Assume there exist  $p,q \in (1, \infty]$ with $$ \frac  1  p +  \frac 1 q < \frac 2 d$$ such that condition $M(p,q)$ is satisfied.
 Then, there exist constants $c_i, \gamma \in (0,\infty)$, $i=11,\ldots, 15$, only depending on $d, p, q, \mean[\om(e)^p]$ and $\mean[\om(e)^{-q}]$, such that, for $\prob$-a.e.\ $\om$,  for all $x \in \bbZ^d$ there exists $N(x, \om) < \infty$ such that for any given $t$ with $\sqrt{t} \geq N(x)$ and all $y \in \bbZ^d$ the following hold.
  \begin{enumerate}
  \item [(i)] If $d_{\om}(x, y)\leq c_{11} t$ then
    \begin{align*}
      p^{\om}(t, x, y)
      \;\leq\;
      c_{12}\, t^{-d/2}\,
      \bigg( 1 + \frac{|x-y|}{\sqrt{t}} \bigg)^{\!\!\ga}\, 
      \exp\!\bigg( \!-\! c_{13}\, \frac{d_{\om}(x,y)^2}{t}\bigg).
    \end{align*}
  \item [(ii)] If $d_{\om}(x,y) \geq c_{15} t$ then
    \begin{align*}
      p^{\om}(t, x, y)
      \;\leq\;
      c_{12}\,  
      \bigg( 1 + \frac{|x-y|}{\sqrt{t}} \bigg)^{\!\!\ga}\,
      \exp\!\bigg( 
        \!-\! c_{14}\, d_{\om}(x,y) 
        \bigg(1 \vee \log \frac{d_{\om}(x,y)}{t}\bigg)
      \bigg).
    \end{align*}
  \end{enumerate}
\end{theorem}
\begin{remark} \label{rem:hke}
  (i) In \cite{ADS16a, ADS19} the heat kernel bounds in Theorems~\ref{thm:hkeCSRW} and \ref{thm:hkeVSRW} are stated for random walks on a general class of graphs including random graphs such as supercritical percolation clusters with long-range correlations, cf.\ Remark~\ref{rem:qip}-(iii) above. For simple random walks on such graphs two-sided Gaussian heat kernel bounds and a parabolic Harnack inequality have been shown in \cite{Sa17}.

  (ii) If the distance $d_{\om}$ and the  Euclidean distance are comparable,  the estimates in Theorem~\ref{thm:hkeVSRW} turn into Gaussian upper bounds since then the additional term $(1+ |x-y|/\sqrt{t})^{\ga}$ can  be absorbed by the exponential term into a constant. Both distances are comparable, for instance,  under i.i.d.\ conductances (cf.\ \cite[Lemma~4.2]{BD10}) or for a class  of models with long-range correlations on $\bbZ^d$  including discrete Gaussian free fields, Ginzburg-Landau $\nabla \phi$-interface models or random interlacements, see \cite{AP24} for details.   However, if both distances are not comparable, the bounds in Theorem~\ref{thm:hkeVSRW} become ineffective in the regime where $d_{\om}(x,y) < \sqrt{t} < |x-y|$, since in this case the term $(1 + |x-y|/\sqrt{t})^{\ga}$ may become large while the exponential term does not provide a decay yet.  

  (iii) The on-diagonal decay $t^{-d/2}$  corresponds to $1/\big|B(x,\sqrt{t})\big|$. In general we expect a stronger decay to hold resulting from the volume of a ball with radius $\sqrt{t}$ w.r.t.\ the distance $d_\om$. For instance, the heat kernel of the aforementioned random walk among random layered conductances admits stronger on-diagonal decay, see \cite{DF20}.

   (iv) For a class of weighted graphs with possibly unbounded geometry, allowing for an unbounded vertex degree, Gaussian upper bounds in terms of intrinsic distances have been shown in \cite{KR24,KR22}, and characterisations of such bounds in terms of Sobolev and Faber-Krahn inequalities have been derived in \cite{KR24a, Ro25}. See also \cite{Ro25a} for a survey.

\end{remark}

Closely related are similar questions about scaling limits and estimates for the associated Green's function. In dimension $d \geq 3$ the behaviour of the Green's function on the full space, defined by $g^{\omega}(x,y) \ldef \int_{0}^{\infty} p_{t}^{\omega}(x, y)\, dt$ is quite well understood. Precise estimates and asymptotics in case of general non-negative i.i.d.\ conductances have been shown in \cite[Theorem~1.2]{ABDH13} and a local limit theorem for $g^{\omega}$ in the case of ergodic conductances under moment conditions in \cite[Theorem~1.14]{ADS16} (cf.also \cite[Theorem~5.2]{Ge20}). In dimension $d = 2$,  precise asymptotics for the associated potential kernel as well as on-diagonal asymptotics and near-diagonal estimates on the Green's function with Dirichlet boundary conditions were obtained in \cite{ADS20}.

\section{Random conductance models with time-dynamic conductances}\label{sec:dynamic}
In this section we discuss homogenization of the dynamic RCM evolving in a time-varying random environment. To define the model, we endow the graph $(\bbZ^d, E_d)$  with a family of time-dependent non-negative conductances $\om = \{\om_t(e) : e \in E_d,\, t \in \bbR \}$.   Let $\Om$ be the set of measurable functions from $\bbR$ to $[0,\infty)^{E_d}$ equipped with a $\si$-algebra $\cF$ and let $\prob$ be a probability measure on $(\Om, \cF)$.   On $\Om$ we define the $d+1$-parameter group of translations $(\tau_{t,x} : (t,x)\in \bbR \times \bbZ^d)$ by
\begin{align*} 
  \tau_{t,x}\!:\Om \rightarrow \Om,
  \qquad
  \big\{\om_s(e) : (s,e) \in \bbR \times E_d\big\}
  \;\longmapsto\;
  \big\{\om_{t+s}(e+x) : (s,e) \in \bbR \times E_d\big\}.
\end{align*}
\begin{assumption}\label{ass:P}
  \begin{enumerate}[(i)]
  \item $\prob$ is ergodic and stationary with respect to space-time shifts, that is, for all $x \in \bbZ^d$, $t\in \bbR$,  $\prob \circ\, \tau_{t,x}^{-1} \!= \prob$, and $\prob[A] \in \{0,1\}\,$ for any $A \in \cF$ such that $\prob[A \triangle \tau_{t,x}(A)] = 0$ for all $x \in \bbZ^d$, $t\in \bbR$.
  
  \item For every $A \in \cF$ the mapping $(\om,t,x)\mapsto \indicator_A(\tau_{t,x}\om)$ is jointly measurable with respect to the $\si$-algebra $\cF \otimes \cB(\bbR)\otimes \cP(\bbZ^d)$.

 % \item $\mean\big[\om_t(e)\big]< \infty$ and $\mean\big[\om_t(e)^{-1}\big] < \infty$ for any $e \in E_d$ and $t \in \bbR$.
  \end{enumerate}
\end{assumption}
For a given $\om \in \Om$ and for $s \in \bbR$ and $x \in \bbZ^d$, let $\Prob_{s,x}^{\om}$ be the probability measure on the space of $\bbZ^d$-valued c\`{a}dl\`{a}g functions on $\bbR$, under which the coordinate process $X = (X_t : t \in \bbR)$ is the time-inhomogeneous Markov process on $\bbZ^d$ starting in $x$ at time $s$ with time-dependent generator acting on  functions $f\!: \bbZ^d \to \bbR$ as
\begin{align*}
  \big(\cL_t^{\om} f\big)(x)
  \;=\;
  \sum_{y: |x-y|=1}\mspace{-6mu} \om_t(\{x, y\}) \, \big(f(y) \,-\, f(x)\big).
\end{align*}
In other words, $X$ is the continuous-time random walk with time-dependent jump rates given by the conductances, i.e.\ the random walk $X$ chooses its next position at random proportionally to the conductances.  Note that the total jump rate out of any lattice site is not normalised, and the law of the sojourn time of $X$ depends on its time-space position, so $X$ is a time-dynamic version of the VSRW having the counting measure as a time-independent invariant measure. For $x,y \in \bbZ^d$ and $t\geq s$, we denote $p^{\om}(s,x;t,y)$ the heat kernel of $(X_t : t \geq s )$, that is
\begin{align*}
  p^{\om}(s,x;t,y)
  \;\ldef\;
  \Prob_{s,x}^{\om}\big[X_t = y\big].
\end{align*}

In order to construct this Markov process under the law $P^\om_{s,x}$, we specify its jump times $s < J_1 < J_2 < \ldots $ inductively. For this purpose, let $\{E_{k} : k\geq 1 \}$ be a sequence of independent $\mathop{\mathrm{Exp}}(1)$-distributed random variables, and set $J_0 = s$ and $X_s = x$. Suppose that for any $k \geq 1$ the process $X$ is constructed on $[s, J_{k-1}]$.  Then, $J_k$ is given by
\begin{align*}
  J_k
  \;=\; J_{k-1} +
  \inf\Big\{%
    t \geq 0 \,:\,
    \int_{J_{k-1}}^{J_{k-1}+t}\! \mu_s^\om(X_{J_{k-1}}) \, \md s \geq E_k
  \Big\},
\end{align*}
and at the jump time $t = J_k$ the random walk $X$ jumps according to the transition probabilities $\{\om_t(X_{J_{k-1}},y)/\mu^\om_t(X_{J_{k-1}}), \, y\sim X_{J_{k-1}}\}$ where we write $\mu^\om_t(x)=\sum_{y\sim x} \om_t(x,y)$, $x\in \bbZ^d$, $t\in \bbR$.
A representative example of the above setting is the VSRW
on dynamical bond percolation on $\bbZ^d$, introduced in \cite{PSS15}.  In this case $\om_t(e)$ is, for each $e\in E_d$, an
independent copy of a stationary continuous-time process taking values in $\{0, 1\}$ with joint invariant measure given by the product of Bernoulli distributions with some prescribed parameter $p \in (0, 1)$. We interpret $\om_t(e) = 1$ as the event that edge $e$ is occupied at time $t$, and $\om_t(e)= 0$ as the event that edge $e$ is vacant. The random walk then jumps at rate $1$ across edges adjacent to its current position that are occupied at that instant of time. If the site where the walk is located has no adjacent occupied edges, then the walk does not move.

For the time-dynamic RCM a QFCLT has been derived under mixing and uniform ellipticity conditions in \cite{An14}. Here we will state a QFCLT under the following significantly weaker assumptions.

\begin{assumption}\label{ass:momentDyn}
  Suppose that either of the following conditions hold.
  \begin{enumerate}[(i)]
    \item There exist $p, q \in (1, \infty]$ satisfying
    \begin{align*}
      \frac{1}{p-1} \,+\, \frac{1}{(p-1) q} \,+\, \frac{1}{q}
      \;<\;
      \frac{2}{d}
    \end{align*}
    such that for any $e \in E_d$ and $t \in \bbR$,
    \begin{align*}
      \mean\!\big[\om_t(e)^p\big] \;<\; \infty
      \quad \text{and} \quad
      \mean\!\big[\om_t(e)^{-q}\big] \;<\; \infty.
    \end{align*} 
    
    \item $t \mapsto  \om_t(e)$ obeys
  \begin{align*}
  \om_t(e) \in [0, 1]
  \end{align*}
  for each $e \in E_d$ and each $t \in \bbR$ and, denoting
  \begin{align*}
    T_e \ldef \inf\bigg\{ t\geq 0: \int_0^t \om_s(e) \, ds \geq 1 \bigg\}, \qquad e\in E_d,
  \end{align*}
 we have
 \begin{align*}
  \exists \vartheta > 4d: \quad \mean\!\big[T_e^\vartheta\big]\; < \;\infty,\qquad  e \in E_d.
 \end{align*}
  \end{enumerate}
  
\end{assumption}

%It is known that under this assumption the process $X$ does not explode, i.e.\ there are only finitely many jumps in finite time, see \cite[Lemma~4.1]{ACDS18}. 

Note that Assumption~\ref{ass:momentDyn}-(ii) allows  $\om_t(e) = 0$ with positive probability, which is ruled out by the moment condition on $1/\om_t(e)$ in Assumption~\ref{ass:momentDyn}-(i). The assumption constitutes a control on the amount of time the conductances may become zero. Note that no condition of the form $\prob(\om_t(e)>0)>p_c$ is needed here. 
Moreover, in contrast to many results on various models for random walks in dynamic random environments, we do not assume the environment to be uniformly elliptic or mixing or Markovian in time and we also do not require any regularity with respect to the time parameter.
\begin{theorem}[QFCLT]\label{thm:dyn_ip}
  Let $d\geq 2$ and suppose that Assumptions~\ref{ass:P} and \ref{ass:momentDyn} hold.  Then, for $\prob$-a.e.\ $\om$, the process $X^{(n)} = \big( n^{-1} X_{n^2 t}: t \geq 0 \big)$ converges, under $\Prob_{\!0,0}^\om$, in law towards a Brownian motion on $\bbR^d$ with a deterministic non-degenerate covariance matrix $\Si^2$.
\end{theorem}
\begin{proof}
  This has been shown in \cite{ACDS18} under Assumption~\ref{ass:momentDyn}-(i)  and in \cite{BR18} under Assumption~\ref{ass:momentDyn}-(ii). 
\end{proof}
The approach for the proofs is largely based on enhanced, but technically involved versions of the methods used for time-static enviroments outlined in Section~\ref{sec:proofCLT}.
Under similar conditions, for one-dimensional random walk under time-dependent degenerate conductances QFCLTs have been derived in \cite{DS16, BP23}, and under the minimal moment assumption of finite first positive and negative moments in \cite{Bi19}.

Under the same moment condition as in Assumption~\ref{ass:momentDyn}-(i) a quenched local limit theorem has also been established.

\begin{theorem}[Quenched local CLT \cite{ACS21}] \label{thm:dyn_lclt} 
  Suppose that Assumptions~\ref{ass:P} and \ref{ass:momentDyn}-(i) hold.  For any $T_2> T_1 > 0$ and $K > 0$,
  \begin{align*}
    \lim_{n \to \infty} \sup_{|x|\leq K} \sup_{ t\in [T_1, T_2]}
    \big| n^d \, p^{\om}(0, 0; n^2 t, \lfloor nx \rfloor) - \bar p^\Sigma(t,0,x) \big|
    \;=\;
    0,
    \qquad \text{for }\prob\text{-a.e. }\om,
  \end{align*}
 where $\Sigma^2$ is the covariance matrix appearing in the limit process in Theorem~\ref{thm:dyn_ip}.
\end{theorem}
Since the static RCM is naturally included in the time-dynamic model, the moment condition in Assumption~\ref{ass:momentDyn} is not optimal for both the QFCLT and the local limit theorem.

Relevant examples for dynamic RCMs include random walks in an environment generated by some interacting particle systems like zero-range or exclusion processes, cf.\ \cite{MO16}, where also some on-diagonal heat kernel upper bounds for a degenerate time-dependent conductances model are obtained via an anchored Nash inequality.  Full two-sided Gaussian estimates are known in the uniformly elliptic case for the VSRW \cite{DD05} or for constant speed walks under effectively non-decreasing conductances \cite{DHZ19}.  However, unlike for static environments, two-sided Gaussian heat kernel bounds are much less regular and some pathologies may arise as they are not stable under perturbations, see \cite{HK16}. Moreover, such bounds are expected to be governed by a time-dynamic version of the intrinsic distance $d_\om$ (cf.\ \eqref{eq:d_om} above) whose exact form in a degenerate setting is unknown. These facts make the derivation of Gaussian bounds for the dynamic RCM with unbounded conductances a subtle open challenge.

\medskip

Random conductance models, in particular their heat kernels and Green's functions,  appear, somewhat unexpectedly, in representations for the correlation functions of various prominent models in modern statistical mechanics.  One prime example is the Ginzburg-Landau $\nabla \phi$ interface model describing, in a particular idealization, an interface between two pure thermodynamical phases; see \cite{Fu05} for a survey.  Here the so-called Helffer-Sj\"{o}strand representation allows to cast the space-time covariances of the height of the surface for convex potentials in terms of the heat kernel of a dynamic RCM (cf.\ \cite{DD05,GOS01}).
However, applications of homogenization results such as FCLTs and local limit theorems in statistical mechanics often require convergence or estimates under the \emph{annealed} measure,  i.e.\ averaged over the law of the environment.  Moreover,  various techniques in quantitative stochastic homogenization theory (cf.\ e.g. \cite{AKM19, GNO15}) rely on annealed estimates.

On one hand,  a QFCLT does imply an annealed FCLT in general.  On the other hand,  the same does not apply to local limit theorems or Gaussian-type estimates for the annealed (or averaged) heat kernel. While the aforementioned analytic Moser and De~Giorgi iteration techniques perform effectively to produce quenched results, one can only derive annealed results from them under stronger and non-explicit moment conditions, see \cite{AT21} for some results in this direction.

In \cite{BDKY15} Benjamini, Duminil-Copin, Kozma and Yadin proposed  a different approach based on a rather
robust entropy method for discrete time random walks on static supercritical percolation clusters. They obtained very sharp annealed results for the discrete gradient of the corresponding heat kernel, which they used in order to prove a Liouville principle for sublinear harmonic functions. 
Recently,  in \cite{DKS23} the entropy method has been adapted by Deuschel, Kumagai and Slowik  to nearest-neighbour RCMs  with time-dependent conductances that are bounded from below but unbounded from above.  Under the minimal moment condition $\mean[\om_t(e)]<\infty$ they obtain sharp annealed on-diagonal estimates for the first and second discrete derivative of the heat kernel,
which they then use  to prove a local limit theorem for the annealed heat kernel and its discrete first derivative as well as optimal decay rates for the annealed Green's function and its derivatives.
Similar results have been obtained in \cite{DD05, MO15} for  uniformly elliptic conductances.
In quantitative stochastic homogenization theory such annealed estimates can be used to control the random part of the homogenization error,  see \cite{MO15}, or to derive quantitative versions of central limit theorems in form of a Berry-Esseen theorem, cf.\ \cite{AN19}.

Finally, let us briefly mention that recent years also witnessed progress in the understanding of certain non-reversible random walks in random environment. One of these are so-called balanced random walks, where the  walk itself is a martingale in every environment. The key technical issue for these is the existence and control of the invariant measure from the environment as seen from the particle. In various settings of both static and time-dynamic environments, invariance principles, local limit theorems and Harnack inequalities have been derived in  \cite{La82,GZ12, BD14,DGR18,DG22, BCDG22,BC22}. 
Another class of treatable jumping mechanisms are random walks in divergence-free (a.k.a.\ doubly stochastic) random environments.  These are exactly the walks for which the underlying shift-invariant law of the environment remains invariant for the environment as seen from the particle.  Under a natural integrability condition of the (static) environment, an annealed FCLT  has been established in \cite{KT17}, and a  QFCLT was subsequently shown in \cite{To18}.
This problem has been studied also in the case when the said integrability condition fails, which leads to superdiffusive behaviour \cite{LTV18,CH-ST22}.

\section{Random conductance models with long-range jumps} \label{sec:longrange}
 In the previous sections our discussion was restricted to RCMs with nearest-neighbour jumps. However, in recent years significant progress has been made in understanding the RCM on $\mathbb{Z}^d$ in the presence of long-range edges. A particular class of resulting underlying (random) graphs is the long-range percolation model in which, for any $x,y\in \mathbb{Z}^d$, an edge between $x$ and $y$ is present with a probability proportional to $|x-y|^{-(d+\alpha)}$, $\alpha>0$. For $\alpha \in (0,2)$, the simple random walk on such graphs is known to be outside the Gaussian universality class, that is the scaling limit is a symmetric $\alpha$-stable L\'{e}vy process instead of Brownian motion, see  \cite{CS13,BT24}, and the heat kernel exhibits non-Gaussian decay \cite{CS12}.
In this regime, similar results including scaling limits, heat kernel bounds and local limit theorems have been derived in \cite{CKW20, CKW21} for RCMs with i.i.d.\ conductances satisfying certain moment conditions.
In \cite{CCK22}, on-diagonal heat kernel bounds were shown and the associated spectral dimensions were identified for simple random walks on long-range percolation graphs in various regimes.

In this section, our discussion will be focused on  a class of random conductance models between the nearest-neighbour and the $\al$-stable regime  ($\alpha> 2$, $d\geq 2$), where the random walk admits long-range jumps but still exhibits a Gaussian large-scale behaviour due to a finite first moment condition on the jump size biased, ergodic random conductances. 

Consider a graph with vertex set $\bbZ^d$, $d \geq 2$, and edge set $\bar{E} = \{\{x,y\} : x, y \in \bbZ^d,\; x \ne y\}$.  Let $(\Om, \cF) = ([0, \infty)^{\bar{E}}, \cB([0,\infty)^{\otimes \bar E}))$ be a measurable space equipped with the Borel-$\si$-algebra.  Similarly as before, for $\om \in \Om$ and $e \in \bar{E}$, we call $\om(e)$ the conductance of the edge $e$ and write $\om(x,y) = \om(y, x) = \om(\{x,y\})$.  Throughout, we will assume that $\mu^{\om}(x) \ldef \sum_{y \in \bbZ^d} \om(x,y)<\infty$ for any $x \in \bbZ^d$. The measure space $(\Om, \cF)$ is naturally equipped with a group of space shifts $\big\{\tau_z : z \in \bbZ^d\big\}$, which act on $\Om$ as
\begin{align*}
  (\tau_z \om)(x, y) \; := \; \om(x+z, y+z),
  \qquad \forall\, \{x,y\} \in  \bar E.
\end{align*}
Henceforth, we consider a probability measure $\prob$ on $(\Om, \cF)$, and we write $\mean$ to denote the expectation with respect to $\prob$.

For any $\om \in \Om$, we consider a variable speed random walk $(X_t : t \geq 0)$ on $\bbZ^d$ with generator $\cL^{\om}_X$ which acts on functions $f\!: \bbZ^d \to \bbR$ with finite support as
\begin{align*} 
  \cL^{\om}_X f(x)
  \;=\;
  \sum_{y \in \bbZ^d} \om(x,y)\, \big( f(y) - f(x) \big).
\end{align*}
  Again we denote by $\Prob_{x}^{\om}$ the law of the process on the space of $\bbZ^d$-valued c\`adl\`ag functions on $\bbR$, starting in $x$.    Moreover, we denote by $p^{\om}(t,x,y) \ldef \Prob_{x}^{\om}[X_t = y]$ for $x, y \in \bbZ^d$ and $t \geq 0$ the heat kernel, that is the transition density with respect to the counting measure as the reversible reference measure. 
\begin{assumption} \label{ass:Plongrange}
 Suppose that $\prob$ satisfies the following conditions.
  \begin{enumerate}[(i)]
    \item $\prob$ is stationary and ergodic with respect to space shifts $\{\tau_x : x \in \bbZ^d\}$ on $\bbZ^d$, i.e.\ $\prob \circ \tau^{-1}_x = \prob$ for all $x \in \bbZ^d$ and $\prob[A] \in \{0,1\}$ for all $A \in \cF$ such that $\tau_x(A) = A$.
    
    \item $\prob\bigl[\mu^{\om}(0) > 0\bigr] = 1$ and $\mean\bigl[\sum_{x \in \bbZ^d} \om(0,x)\, |x|^2 \bigr] < \infty$.
    
    \item $\prob$ is irreducible in the sense that
    \begin{align*}
      \prob\Bigl[
        \bigl\{ 
          \om \,:\, \sup\nolimits_{t \geq 0}\, p^{\om}(t,0,x) > 0
        \bigr\}
      \Bigr]
      \;=\;
      1,
      \qquad \forall\, x \in \bbZ^d.
    \end{align*}
  \end{enumerate}
\end{assumption}

\begin{assumption} \label{ass:momentLongrange}
  There exist $p, q \in (1, \infty]$ satisfying
  \begin{align*}
    \frac{1}{p} \,+\, \frac{1}{q}
    \;<\;
    \frac{2}{d}
  \end{align*}
  such that
  \begin{align}  \label{eq:moment_pq}
    \mean\biggl[
      \Bigl( {\textstyle \sum_{x \in \bbZ^d}\;} \om(0,x)\, |x|^2  \Bigr)^p
    \biggr]
    \;<\; 
    \infty
    \quad \text{and} \quad
    \mean\bigl[\om(0,x)^{-q}\bigr] \;<\; \infty
  \end{align}
  for any $x\in \bbZ^d$ with $|x|=1$. In particular, $\om(x,y)>0$ $\prob$-a.s.\ whenever $x\sim y$. 
\end{assumption}
Assumption~\ref{ass:momentLongrange} is a direct extension of the moment condition $M(p,q)$ for $p,q\geq 1$ with $1/p+1/q<2/d$, see Definition~\ref{def:Mpq}, originally introduced in \cite{ADS15} for nearest-neighbour RCMs.
\begin{theorem}[QFCLT \cite{BCKW21}]\label{thm:ipLongrange}
  Suppose that Assumptions~\ref{ass:Plongrange} and \ref{ass:momentLongrange} hold.  Then, for $\prob$-a.e.\ $\om$, the process $X^{(n)} = \bigl(n^{-1} X_{n^2 t} : t \geq 0 \bigr)$ converges, under $\Prob_{\!0}^{\om}$, in law towards a Brownian motion on $\bbR^d$ with a deterministic non-degenerate covariance matrix $\Si^2$.
\end{theorem}
Similarly as the QFCLTs in the previous sections, the proof of Theorem~\ref{thm:ipLongrange} relies on  elliptic regularity techniques. However, unlike those results, elliptic regularity comes in the form of exit time estimates, obtained by truncation to finite volume, since iterative techniques such as the usual Moser iteration do not perfom well in the long-range setting as we will explain below. The approach can also be used show FCLTs for random walks under deterministic conductance configurations \cite{Bi23}.

%Under a weaker moment condition, namely Assumption~\ref{ass:momentLongrange} with $p = q = 1$, a central limit theorem for $X$ has recently been obtained in \cite{Fa23}. 
Spectral homogenization properties of the discrete Laplace operator with random conductances satisfying Assumptions~\ref{ass:Plongrange} and \ref{ass:momentLongrange} have been studied in \cite{FHS19}, where  almost-sure homogenization of the discrete Poisson equation and of the top of the Dirichlet spectrum is obtained.

\begin{example}[Long-range percolation]
Theorem~\ref{thm:ipLongrange} applies to random walks on a family of long-range percolation graphs.  These graphs are obtained by taking $\bbZ^d$ as an underlying graph and adding edges independently with a probability that depends only on the distance between the endpoints.  While this probability is typically assumed to decay as a power of the distance, the formulation of the moment condition in Assumption~\ref{ass:momentLongrange} only requires a summability condition.  More precisely, let $p\!: \bbZ^d \to [0,1]$ be a function such that
  \begin{enumerate}
    \item[(i)] $p(x) = p(-x)$ for all $x \in \bbZ^d$,
    \item[(ii)] $p(0) = 0$ and $p(x) = 1$ whenever $|x|=1$,
      \item[(iii)] $\sum_{x \in \bbZ^d} p(x) |x|^{2p} < \infty$ for some $p > d/2$.
  \end{enumerate}
  For any $\{x, y\} \in \bar{E}$ we define the random conductances, $\om(x,y)$, only taking values in $\{0, 1\}$, by setting $\om(x,y) = 1$ with probability $p(y-x)$ and $\om(x,y) = 0$ otherwise, independently of all other edges.  Then Assumptions~\ref{ass:Plongrange} and \ref{ass:momentLongrange} are satisfied, see \cite[Corollary~2.3]{BCKW21} .  Finally, in the context of power-law decaying connection probabilities, condition~(iii) can be rephrased as $p(x) = |x|^{-s + o(1)}$ for $|x| \to \infty$ for some $s > 2d$.   We refer to \cite[Section~2.2]{BCKW21} for a more detailed discussion of this example.
\end{example}
The corresponding local limit theorem has recently been shown under the following stronger moment condition.

\begin{assumption} \label{ass:momentLongrange2}
  There exist $p, q \in (1, \infty]$ satisfying
  \begin{align*}
    \frac{1}{p} \,+\, \frac{1}{q}
    \;\leq\;
   \bigg(1 + \frac 1 p \bigg) \frac{1}{d}, \qquad \frac 1 {p-1}+\frac 1 q \; < \; \frac 2 d,
  \end{align*}
  such that \eqref{eq:moment_pq} is satisfied,  and that either the following condition holds, 
\begin{align*}
  \mean\biggl[
    \Bigl( {\textstyle \sum_{x \in \bbZ^d}\;} \om(0,x)\, |x|^{d+2}  \Bigr)^p
  \biggr] \; < \; \infty,
\end{align*}
or $q = + \infty$ which is equivalent to $$\inf_{x,y\in \bbZ^d: |x-y|=1} \om(x,y)>0.$$
\end{assumption}

\begin{theorem}[Quenched local CLT \cite{CKW24}] \label{thm:lclt_longrange} 
  Suppose that Assumptions~\ref{ass:Plongrange} and \ref{ass:momentLongrange2} hold. 
  Then, for any $T_2> T_1 > 0$ and $K > 0$,
  \begin{align*}
    \lim_{n \to \infty} \sup_{|x|\leq K} \sup_{ t\in [T_1, T_2]}
    \big| n^d \, p^{\om}(0, n^2 t, \lfloor nx \rfloor) - \bar p^\Sigma(t,0,x) \big|
    \;=\;
    0,
    \qquad \text{for }\prob\text{-a.e. }\om,
  \end{align*}
 where $\Sigma^2$ is the covariance matrix appearing in the limit process in Theorem~\ref{thm:dyn_ip}.
\end{theorem}

Similarly as for the local limit theorems for the nearest neighbour RCMs discussed above, the main step in the proof of Theorem~\ref{thm:lclt_longrange} is to establish a weak parabolic Harnack inequality, from which an oscillation inequality and a H\"older regularity estimate for non-local space-time harmonic function can be deduced. The approach requires a maximal inequality for such functions, which is possible to establish, for instance by Moser or De~Giorgi's iteration techniques. However, one fundamental difference from the analysis of equations with local discrete finite-difference operators, appearing in the nearest-neighbour case, is the presence of a non-local tail term of the harmonic function, see for instance \cite{dCKP14, dCKP16, Str19, KW22a, KW24}, and so far the available analytic methods to control the tail term are mainly restricted to the setting of uniformly elliptic weights. To our knowledge, \cite{CKW24} is one of the first works providing a control on the tail term in a setting with unbounded weights, albeit requiring a stronger moment condition.  This makes the derivation of a local limit theorem and heat kernel estimates under optimal moment conditions appear to be a challenging open problem.

\bigskip 
\subsection*{Acknowledgement}
The author is grateful to the organisers of the conference Fractal Geometry
and Stochastics~7, held at TU~Chemnitz in September 2024,  for arranging a nice meeting and for inviting him to produce this article. He also thanks Martin Barlow, Marek~Biskup, David~Croydon and Christian~Rose for valuable comments and suggestions, and Martin Slowik for providing the simulations in Figure~\ref{fig:numsim}.

\bibliographystyle{abbrv}
\bibliography{literature_survey}

\end{document}